\documentclass[12pt,twoside,english]{article}
\usepackage{amsfonts,babel,amssymb,amsmath,amsthm}
\usepackage{graphicx}

\newtheorem{proposition}{Proposition}[section]
\newtheorem{corollary}[proposition]{Corollary}
\newtheorem{lemma}[proposition]{Lemma}
\newtheorem{theorem}[proposition]{Theorem}

\newtheorem{remark}{Remark}[section]

\numberwithin{equation}{section}

\newcommand{\punto}{\,\,\cdot\,\,}
\newcommand{\ds}{\displaystyle}
\newcommand{\smallfrac}[2]{{\textstyle\frac{#1}{#2}}}
\newcommand{\jump}[1]{[\![#1]\!]}
\newcommand{\triple}[1]{|\!|\!|#1|\!|\!|}

\title{Fully discrete Kirchhoff formulas with CQ--BEM}

\author{Lehel Banjai
\footnote{School of Mathematical \& Computer Sciences, Heriot-Watt University, Edinburgh, EH14 4AS, United Kingdom -- {\tt l.banjai@hw.ac.uk}}
Antonio R. Laliena
\footnote{Dep.~Matem\'aticas, EUPLA, Universidad de Zaragoza, 50100 La Almunia, Spain --
{\tt  arlalibi@unizar.es}. Partially supported by MICINN Project MTM2010-16917 and Gobierno de Arag\'on (Grupo consolidado PDIE).}
 \& Francisco--Javier Sayas\footnote{Department of Mathematical Sciences, University of Delaware, Newark DE, 19716 USA -- {\tt fjsayas@udel.edu}. Partially supported by the National Science Foundation (grant DMS-1216356).}
}

\date{\today}

\begin{document}

\maketitle

\begin{abstract}
In this paper we propose and analyze a fully discrete method for a direct boundary integral formulation of the scattering of a transient acoustic wave by a sound-soft obstable. The method uses Galerkin-BEM in the space variables and three different choices of time-stepping strategies based on Convolution Quadrature. The numerical analysis of the method is carried out directly in the time domain, not reverting to Laplace transform techniques.



\end{abstract}

\section{Introduction}

In this paper we propose and analyze a fully discrete method for the direct boundary integral formulation of the Dirichlet problem for the causal acoustic wave equation, exterior to a domain with Lipschitz boundary in $\mathbb R^d$ ($d=2$ or $3$). The method arises from using a general Galerkin semidiscretization-in-space and multistep-based Convolution Quadrature (CQ) in time. From the point of view of the numerical method, this paper extends work in \cite{Lub94} and \cite{Ban10}. A survey of recent results for CQ-BEM discretization of a wide variety transient problems can be found in \cite{BanSchTA}.

Analytical literature on time-domain integral equations has typically been focused on integral equations of the first kind arising from indirect formulations. The origin of these techniques was based on Galerkin-in-time methods \cite{BamHaD86a}, while CQ techniques were developed only about one decade later. The present paper uses a direct formulation, leading to an integral equation of the first kind similar to those treated in \cite{BamHaD86a, Lub94, Ban10}. The main differences lie in the fact that data appear under the action of a retarded integral operator (which will have to be discretized as well) and that the unknown on the boundary is a quantity of physical interest. 

In this paper we propose the development of a systematic analysis of CQ-BEM taking care of all aspects of discretization: (a) data interpolation on the computational grid, (b) Galerkin semidiscretization-in-space of the associated retarded integral equation, (c) discretization in time (using CQ) of the integral operators in both sides of the equation, (d) discretization in time of the postprocessed potentials (acting on the data and unknown of the integral equation) to obtain the scattered wave field in exterior points. The main difference with the traditional black-box analysis proposed by Lubich \cite{Lub94} is in the fact that we propose to do most of the analysis directly in the time domain. Original work in the analysis of CQ-BEM dealt only with the simplest retarded boundary integral equations, that are coercive in the Laplace domain. Coercivity is inherited by the Galerkin semidiscrete-in-space problem, but some properties of the fully discrete problem (including postprocessing of the solution to obtain the associated potentials and treatment of data that appear under the action of retarded integral operators) have to be investigated in a more direct fashion \cite{LalSay09}. More recently, some estimates in the time-domain \cite{DomSaySB, SaySB} have expanded the analytical toolbox that can be used to prove error estimates for full discretization of retarded boundary integral equations. It has to be noted that most of the literature that is relevant for this analysis had been carried out using Laplace transforms --the paper \cite{Ryn99} seems to be a lone exception--. The passage through the Laplace domain makes for a relatively streamlined analysis that can be expanded to a wide variety of problems \cite{LalSay09, BanSauSB} but is likely to yield less sharp results than a direct analysis in the time domain. As announced, in this paper we will develop the more recent technology of time-domain estimates to show properties of the Galerkin semidiscrete-in-space problem ({\em these are pertinent for other kind of time-discretization methods}, using Galerkin schemes) and of the full discretization of the problem. In particular, we will obtain a proof of convergence of the trapezoidal rule CQ method that is not directly reachable in the Laplace domain. The tools for this analysis are varied but not complicated: (a) identification of the weak convolutional retarded integral equations and layer potentials with strong solutions of problems in finite domains for finite time intervals, (b) interpretation of Galerkin semidiscretization-in-space with exotic transmission problems following \cite{LalSay09, SaySB}, (c) use of the well understood theory of $C_0$-groups of isometries \cite{Paz83} to obtain estimates for the resulting dynamical systems \cite{DomSaySB}, (d) understanding of the process of CQ time-discretization as a direct discretization of the exotic transmission problems in the time-domain (the essence of this idea is already present but not exploted in \cite{Lub94}) and (e) application of 
standard techniques for numerical analysis of the wave equation in bounded domains
to work out the analysis of the fully discrete method .

\paragraph{Foreword.} Elementary properties of basic Sobolev spaces $H^1(\Omega)$ and $H^{\pm1/2}(\Gamma)$, the trace operator and the weak normal derivative, will be used without further reference. The pertinent results can be found in any advanced textbook of elliptic PDE. The monograph \cite{McL00} contains all of them, as well as some results about steady-state layer potentials and integral operators that will be similar to the ones we will be developing in this paper (and that are used to prove background results that will be explicitly mentioned as they are used). While possible, we will make an effort in clarifying the source of constants in estimates. Once this is not practical any more, we will use the convention of admitting $C>0$ to be a constant indendendent of the associated discretization parameters ($h$ and $\kappa$ in this paper). Vector-valued distributions appear in the background of the theory of retarded layer potentials and integral operators. Their use  has been outsourced to some preliminary papers \cite{LalSay09, DomSaySB, SaySB} that relate strong and weak solutions of the wave equation. Here we will only employ the basic idea of a {\em causal distribution} with values on a space $X$ as a sequentially bounded map $\mathcal D(\mathbb R) \to X$ that vanishes when applied to elements of $\mathcal D((-\infty,0)$. The concepts of differentiation and Laplace transform are then identical to those of scalar distributions.

\section{An integral formulation of the scattering problem}\label{sec:2}

Let $\Omega^-$ be a bounded open set in $\mathbb R^d$ (with $d=2$ or
$3$) with Lipschitz boundary $\Gamma$. We admit the possibility that
$\Omega^-$ is not connected, but we demand the complementary domain
$\Omega^+:=\mathbb R^d\setminus\overline\Omega^-$ to be connected. The
problem of scattering of an incident wave by a sound-soft obstacle
can be written by means of the Initial Boundary Value Problem
\begin{subequations}\label{eq:2.1}
\begin{alignat}{4}
u_{tt}=\Delta u & \qquad & & \mbox{in $\Omega^+$, $\forall t>0$},\\
u+u^{\mathrm{inc}}=0 & & & \mbox{on $\Gamma$, $\forall t>0$},\\
u(\punto,0)=u_t(\punto,0)=0 & & & \mbox{on $\Gamma$}.
\end{alignat}
\end{subequations}
The incident wave $u^{\mathrm{inc}}$ is a known function. For the
model equation to be meaningful we have to assume that
$u^{\mathrm{inc}}(\punto,0)\equiv u_t^{\mathrm{inc}}(\punto,0)\equiv
0$ in a neighborhood of $\Gamma$. The unknown in \eqref{eq:2.1} is
the scattered wave field, while the total wave is
$u+u^{\mathrm{inc}}$. There is no need to impose a radiation
condition at infinity since causality of the wave equation takes
care of the fact that the support of the solution of \eqref{eq:2.1},
for any given $t$, is compact.

For all purposes (expository and analytic), it is convenient to
understand functions of the space and time variables as functions of
$t$ with values on a space of function. Therefore, instead of
considering $u=u(\mathbf x,t)$ as a function in $\Omega^+ \times
[0,\infty)$, we will consider $u=u(t)$, where $u(t)\in
H^1(\Omega^+)$ for all $t$. It will also be convenient to refer to
{\em causal functions} as functions $\xi:\mathbb R \to X$ (here $X$ is
any Hilbert space) such that $\xi(t)=0$ for all $t< 0$. The
concept of causality can be easily extended to distributions with
values in the space $X$.

An integral representation of the solution of \eqref{eq:2.1} starts
by taking the value of the incident wave on $\Gamma$ for positive
values of the time variable. If $\gamma^\pm:H^1(\Omega^\pm) \to
H^{1/2}(\Gamma)$ are the trace operators on $\Gamma$, we consider
the causal function
\begin{equation}\label{eq:2.0}
\varphi:\mathbb R \to H^{1/2}(\Gamma) \qquad \mbox{such that}\qquad
\varphi(t)=\gamma^+ u^{\mathrm{inc}}(\punto,t) \qquad \forall t>0.
\end{equation}
Note that the required regularity of the incident wave for this
process to be meaningful is local $H^1$ behavior in a neighborhood
of $\Gamma$ and that $u^{\mathrm{inc}}$ can have singularities away
from the scattering boundary (this is the case for waves originated
by acoustic sources).

Consider now the single and double layer retarded acoustic
potentials. Their strong expressions in the three dimensional case
(valid for smooth-in-space densities written as functions of the
space and time variables) are
\begin{equation}\label{eq:2.2}
(\mathcal S*\lambda)(\mathbf x,t):=\int_\Gamma \frac{\lambda(\mathbf
y,t-|\mathbf x-\mathbf y|)}{4\pi|\mathbf x-\mathbf y|}\,\mathrm
d\Gamma(\mathbf y) \qquad \mathbf x\in \mathbb R^3\setminus\Gamma,
\end{equation}
and
\begin{equation}\label{eq:2.3}
(\mathcal D*\varphi)(\mathbf x,t):=\int_\Gamma \nabla_{\mathbf
y}\left(\frac{\varphi(\mathbf z,t-|\mathbf x-\mathbf y|)}{4\pi |\mathbf
x-\mathbf y|}\right)\Big|_{\mathbf z=\mathbf y}\cdot
\boldsymbol\nu(\mathbf y)\, \mathrm d\Gamma(\mathbf y), \qquad
\mathbf x\in \mathbb R^d\setminus\Gamma,
\end{equation}
respectively.
In \eqref{eq:2.3} the vector $\boldsymbol\nu(\mathbf y)$ denotes the
unit outward pointing normal vector at the point $\mathbf y\in
\Gamma$. The notation for the layer potentials in \eqref{eq:2.2} and
\eqref{eq:2.3} uses the convolutional symbol to emphasize the fact
that these are time-convolution operators (see \cite{LalSay09b} for
a rigorous introduction of these operators in the sense of
distributions), since we will take advantage of this convolutional
structure for the discretization in the time variable.

By using direct arguments in the time domain 
\cite{LalSay09b} or employing Laplace transforms 
\cite{BamHaD86a, BamHaD86b}, it is possible to prove that if
$\lambda$ is a causal distribution with values in the space
$H^{-1/2}(\Gamma)$, then $\mathcal S*\lambda$ is a causal
distribution with values in the space
\begin{equation}\label{eq:2.4}
H^1_\Delta (\mathbb R^d\setminus\Gamma):= \{ v \in H^1(\mathbb
R^d\setminus\Gamma)\,:\, \Delta u\in L^2(\mathbb
R^d\setminus\Gamma)\},
\end{equation}
satisfying
\begin{equation}\label{eq:2.5}
\Delta (\mathcal S*\lambda)=\smallfrac{\mathrm d^2}{\mathrm
dt^2}(\mathcal S*\lambda) \qquad \mbox{in $\mathbb
R^d\setminus\Gamma$}
\end{equation}
and
\begin{equation}\label{eq:2.6} \jump{\gamma(\mathcal
S*\lambda)}:=\gamma^+ (\mathcal S*\lambda)-\gamma^+(\mathcal
D*\lambda)=0, \qquad \jump{\partial_\nu (\mathcal
S*\lambda)}:=\partial_\nu^+ (\mathcal
S*\lambda)-\partial_\nu^-(\mathcal S*\lambda)=\lambda.
\end{equation}
Note that the Laplace operator (in the sense of distributions in
$\mathbb R^d\setminus\Gamma$) and the exterior and interior normal
derivatives are well defined in the space $H^1_\Delta(\mathbb
R^d\setminus\Gamma)$. Expressions \eqref{eq:2.5} and
\eqref{eq:2.6} can be understood as equalities of causal 
distributions with values in $L^2(\mathbb R^d\setminus\Gamma)$,
$H^{1/2}(\Gamma)$ and $H^{-1/2}(\Gamma)$ respectively. The second
derivative in \eqref{eq:2.5} is defined in the sense of vector
valued distributions. The first of the jump relations \eqref{eq:2.6}
allows us to define the retarded integral operator
\begin{equation}\label{eq:2.7}
\mathcal V*\lambda:=\gamma^+(\mathcal S*\lambda)=\gamma^-(\mathcal
S*\lambda),
\end{equation}
whose integral expression in the three dimensional case (for smooth
enough densities) coincides with that of $\mathcal S*\lambda$, with
$\mathbf x\in \Gamma$ now.

If $\varphi$ is a causal distribution with values in
$H^{1/2}(\Gamma)$, then $\mathcal D*\varphi$ is a causal 
distribution with values in $H^1_\Delta (\mathbb
R^d\setminus\Gamma)$ satisfying
\begin{equation}\label{eq:2.8}
\Delta (\mathcal D*\varphi)=\smallfrac{\mathrm d^2}{\mathrm
dt^2}(\mathcal D*\varphi) \qquad \mbox{in $\mathbb R^d\setminus\Gamma$}
\end{equation}
and
\begin{equation}\label{eq:2.9}
\jump{\gamma(\mathcal D*\varphi)}=-\varphi, \qquad
\jump{\partial_\nu(\mathcal D*\varphi)}=0.
\end{equation}
We then define the retarded boundary integral operator
\begin{equation}\label{eq:2.9b}
\mathcal K*\xi=\smallfrac12 \gamma^+(\mathcal
D*\xi)+\smallfrac12\gamma^-(\mathcal D*\xi).
\end{equation}
An integral expression for this operator in the three dimensional
case coincides with that of the layer operator $\mathcal D*\xi$ (see
\eqref{eq:2.3}). Any causal $H^1_\Delta(\mathbb
R^d\setminus\Gamma)$-valued tempered distribution $u$ such that
\[
\ddot u=\Delta u \qquad \mbox{in $\mathbb R^d\setminus\Gamma$}
\]
(with equality as $L^2(\mathbb R^d\setminus\Gamma)$-valued
distributions and with the usual notation $\ddot u=\frac{\mathrm d^2
u}{\mathrm dt^2}$) can be represented with Kirchhoff's formula
\[
u=\mathcal S*\jump{\partial_\nu u}-\mathcal D*\jump{\gamma u}.
\]
Therefore, if we consider a solution of \eqref{eq:2.1} as a causal
tempered $H^1(\Omega^+)$-valued distribution that is extended by zero to
$\Omega^-$ and denote $\varphi$ as in \eqref{eq:2.0}, we can write
\begin{equation}\label{eq:2.10}
u=-\mathcal S*\lambda -\mathcal D*\varphi
\end{equation}
where $\lambda:=\partial_\nu^+u$. Using the definitions of the
boundary operators in \eqref{eq:2.7} and \eqref{eq:2.9b} as well as
the first of the jump relations \eqref{eq:2.9}, it follows that
$\lambda$ satisfies the following equation
\begin{equation}\label{eq:2.11}
\mathcal V*\lambda= \smallfrac12\varphi-\mathcal K*\varphi.
\end{equation}
The analysis of \cite{BamHaD86a} includes a proof of the unique
solvability of the operator equation in \eqref{eq:2.11} and a
Sobolev estimate for the solution of $\mathcal V*\xi=\varphi$. Also \cite[Theorem 6.2]{DomSaySB} contains an estimate of the solution operator for equation \eqref{eq:2.11} and its postprocessing \eqref{eq:2.10}.

\section{Discretization}\label{sec:3}

We start by assuming that the data function $\varphi$ has been
approximated. This is the usual approach of the engineering literature (see the exposition of a very similar family of methods for elastic waves in \cite{Sch01} for instance) and will be for us a motive to studying the propagation of errors in data, a study that will be needed for analysis of the fully discrete schemes.
We therefore assume that a
causal function $\varphi_h :\mathbb R \to H^{1/2}(\Gamma)$ is given
as an approximation to $\varphi$.

\subsection{Semidiscretization in space}

We consider a discrete space $X_h \subset
L^\infty(\Gamma)$ and substitute \eqref{eq:2.11} by the search of a
causal function $\lambda_h : \mathbb R \to X_h$ such that
\begin{equation}\label{eq:3.1}
\langle \mu_h, \mathcal V*\lambda_h\rangle_\Gamma = \langle\mu_h
,\smallfrac12\varphi_h-\mathcal K*\varphi_h \rangle_\Gamma \qquad
\forall \mu_h \in X_h.
\end{equation}
Here and in the sequel the angled brackets denote the
$H^{-1/2}(\Gamma) \times H^{1/2}(\Gamma)$ duality product. The
solution of \eqref{eq:3.1} is then used for the discrete
representation formula
\begin{equation}\label{eq:3.2}
u_h:=-\mathcal S*\lambda_h-\mathcal D*\varphi_h.
\end{equation}

For the sake of clarity, let us write down  the system
\eqref{eq:3.1}  in the three dimensional case, when data
have been approximated by a function $\varphi_h :\mathbb R \to
Y_h\subset W^{1,\infty}(\Gamma)$, where $Y_h$ is finite dimensional.
Let $\{ N_1,\ldots,N_J\}$ and $\{ M_1,\ldots,M_K\}$ be respective
bases of $X_h$ and $Y_h$. Data and unknown can then be represented
by their coefficients:
\begin{equation}\label{eq:3.bases}
\varphi_h(t)= \sum_{k=1}^K \varphi_k(t) M_k, \qquad \lambda_h(t) =
\sum_{j=1}^J \lambda_j(t) N_j.
\end{equation}
This is equivalent to substituting the $Y_h$- and $X_h$-valued
functions by a finite set of casual scalar functions. Problem
\eqref{eq:3.1} is equivalent to the system
\begin{eqnarray}\nonumber
& & \hspace{-2cm} \sum_{j=1}^J\int_\Gamma \int_\Gamma
\frac{N_i(\mathbf x)N_j(\mathbf y)}{4\pi |\mathbf x-\mathbf y|}
\lambda_j(t-|\mathbf x-\mathbf y|) \mathrm d\Gamma(\mathbf x)\mathrm
d\Gamma(\mathbf y)
\\
\nonumber & =& \frac12 \sum_{k=1}^K \left(\int_\Gamma N_i(\mathbf
x)M_k(\mathbf x)\mathrm d\Gamma (\mathbf x)\right) \varphi_k(t)\\
& & +\sum_{k=1}^K \int_\Gamma\int_\Gamma \mathrm M_{i,k}(\mathbf
x,\mathbf y)\Big(\frac{\varphi_k(t-|\mathbf x-\mathbf y|)}{|\mathbf
x-\mathbf y|}+\dot \varphi_k(t-|\mathbf x-\mathbf y|)\Big)\mathrm
d\Gamma(\mathbf x)\mathrm d\Gamma(\mathbf y)\label{eq:3.3}
\end{eqnarray}
(for $i=1,\ldots,J$), where
\[
\mathrm M_{i,k}(\mathbf
x,\mathbf y):= \frac{(\mathbf x-\mathbf
y)\cdot\boldsymbol\nu(\mathbf y)}{4\pi|\mathbf x-\mathbf y|^2}
N_i(\mathbf x)M_k(\mathbf y).
\]
The following spaces will be relevant in the sequel:
\begin{eqnarray*}
\mathcal W^k_0(\mathbb R;X) &:=&\{ \rho\in \mathcal C^{k-1}(\mathbb R;X) \,:\, \mathrm{supp}\,\rho \subset [0,\infty),\quad \rho^{(k)}\in L^1_{\mathrm{loc}}(\mathbb R;X)\},\\
\mathcal C^k_0(\mathbb R;X) &:=& \{ \rho\in \mathcal C^{k}(\mathbb R;X) \,:\, \mathrm{supp}\,\rho \subset [0,\infty)\}.
\end{eqnarray*}

\begin{theorem}\label{th:3.1}
Let $\varphi \in \mathcal W^4_0(\mathbb R;H^{1/2}(\Gamma))$, assume that the solution of \eqref{eq:2.11} satisfies $\lambda \in \mathcal W^2_0(\mathbb R;H^{-1/2}(\Gamma))$ and let $u$ be given by \eqref{eq:2.10}. Assume also that $\varphi_h \in \mathcal W^4_0(\mathbb R; Y_h)$. Then the semidiscrete equation \eqref{eq:3.1} has a unique solution $\lambda_h$. Let finally $u_h$ be given by \eqref{eq:3.2}. Then, for all $t\ge 0$,
\begin{eqnarray}\nonumber
\|(\lambda-\lambda_h)(t)\|_{-1/2,\Gamma} & \le & C (1+t) \Big( \sum_{\ell=0}^2 \int_0^t \|(\lambda^{(\ell)}-\Pi_h\lambda^{(\ell)})(\tau)\|_{-1/2,\Gamma}\mathrm d\tau\\
\label{eq:3.A}
& & \hspace{2cm} +\sum_{\ell=0}^4 \| (\varphi^{(\ell)}-\varphi_h^{(\ell)})(\tau)\|_{1/2,\Gamma}\mathrm d\tau\Big)\\
\nonumber
\| (u-u_h)(t)\|_{1,\mathbb R^d}  & \le & C (1+t) \Big( \sum_{\ell=0}^2 \int_0^t \|(\lambda^{(\ell)}-\Pi_h\lambda^{(\ell)})(\tau)\|_{-1/2,\Gamma}\mathrm d\tau\\
\label{eq:3.B}
& & \hspace{2cm} +\sum_{\ell=0}^2 \| (\varphi^{(\ell)}-\varphi_h^{(\ell)})(\tau)\|_{1/2,\Gamma}\mathrm d\tau\Big),
\end{eqnarray}
where $\Pi_h:H^{-1/2}(\Gamma)\to X_h$ is the orthogonal projection onto $X_h$. 
\end{theorem}

Note that by \cite[Theorem 6.2]{DomSaySB}, if $\varphi \in \mathcal W^4_0(\mathbb R;H^{1/2}(\Gamma))$, then $\lambda \in \mathcal C^0_0(\mathbb R;H^{-1/2}(\Gamma))$. Time regularity of the solution is then guaranteed by the sufficient (but not necessary) condition $\varphi \in \mathcal W^6_0(\mathbb R;H^{1/2}(\Gamma))$. The proof of Theorem \ref{th:3.1} will be given at the end of Section \ref{sec:5}. If data are not discretized the second group of terms in the error estimates of Theorem \ref{th:3.1} is not needed. 

\subsection{Full discretization}

In a final step, we substitute the four time convolutions that
appear in \eqref{eq:3.1} and \eqref{eq:3.2} with a discrete
convolution based on one of the applicable Convolution Quadrature
methods. For a fixed time-step $\kappa>0$, the CQ method applied to
the discretization of \eqref{eq:3.1} and \eqref{eq:3.2} produces (in
theory) casual functions $\lambda_h^\kappa :\mathbb R \to X_h$ and
$u_h^\kappa:\mathbb R \to H^1_\Delta(\mathbb R^d\setminus\Gamma)$.
In practice, these functions are evaluated in equally spaced time steps $t_n:=
n\,\kappa$ and only these values of the functions are obtained. To
obtain values at other times, the method has to be run again,
starting at $t_0:=-\varepsilon \, \kappa$ for instance. Therefore,
even if the theory of CQ deals with functions of continuous time, in
practice the solutions can be understood as functions of discrete
time, i.e., sequences.

Let us briefly explain what the CQ discretization of
\eqref{eq:3.3} consists of. First of all, we consider the complex
matrix valued functions $\mathbb V_h(s) \in \mathbb R^{J\times J}$
and $\mathbb K_h(s) \in \mathbb R^{J\times K}$ with elements
\begin{eqnarray}\label{eq:3.10}
\mathbb V_h(s)_{i,j} &:=& \int_\Gamma\int_\Gamma \frac{N_i(\mathbf
x)N_j(\mathbf y)}{4\pi |\mathbf x-\mathbf y|} e^{-s|\mathbf
x-\mathbf y|}\mathrm d\Gamma(\mathbf x)\mathrm d\Gamma(\mathbf y)\\
\label{eq:3.11}
\mathbb K_h(s)_{i,j} &:=& \int_\Gamma\int_\Gamma \Phi_{i,k}(\mathbf
x,\mathbf y) \left(1+s|\mathbf x-\mathbf
y|\right)\frac{e^{-s|\mathbf x-\mathbf y|}}{|\mathbf x-\mathbf
y|}\mathrm d\Gamma(\mathbf x)\mathrm d\Gamma(\mathbf y)
\end{eqnarray}
and the matrix with elements
\[
\mathbb I_{i,k}:=\int_\Gamma N_i(\mathbf x)M_k(\mathbf x)\mathrm
d\Gamma (\mathbf x).
\]
Note that $\mathbb V_h(s)$ and $\mathbb K_h(s)$ are the Laplace
transforms of the operators that appear in \eqref{eq:3.3}. We then
construct the Taylor expansions
\begin{equation}\label{eq:3.4}
\mathbb V_h(\kappa^{-1}\delta(\zeta))=\sum_{n=0}^\infty \mathbb
V_h^\kappa [n] \zeta^n \qquad \mathbb
K_h(\kappa^{-1}\delta(\zeta))=\sum_{n=0}^\infty \mathbb K_h^\kappa
[n] \zeta^n,
\end{equation}
where $\delta$ is one of the following functions
\[
\delta(\zeta):= \left\{ \begin{array}{ll}1-\zeta, & \mbox{(backward Euler method)}\\ 
\frac32-2\zeta+\frac12\zeta^2 & \mbox{(BDF2)}\\
2\,\frac{1-\zeta}{1+\zeta} &\mbox{(trapezoidal rule)}.
\end{array}\right.
\]
 Data
discretization consists of the construction of vectors
$\boldsymbol\varphi_h^\kappa[n]:=(\varphi_k(n\,\kappa))_{k=1}^K$ for $n\ge 0$ (recall \eqref{eq:3.bases}).
The unknowns are vectors $\boldsymbol\lambda_h^\kappa[n]\in \mathbb
R^J$ satisfying
\begin{equation}\label{eq:3.5}
\sum_{m=0}^n \mathbb
V_h^\kappa[m]\boldsymbol\lambda_h^\kappa[n-m]=\smallfrac12 \mathbb
I_h \boldsymbol\varphi_h^\kappa[n]+\sum_{m=0}^n \mathbb
K_h^\kappa[m]\boldsymbol\varphi_h^\kappa[n-m], \qquad n\ge 0.
\end{equation}
We can thus associate
\[
\boldsymbol\lambda_h^\kappa[n]=(\lambda_1[n],\ldots,\lambda_J[n])\qquad\longmapsto\qquad
\lambda_h^\kappa[n]:=\sum_{j=1}^J \lambda_j[n] N_j \in X_h
\]
to obtain a fully discrete approximation of $\lambda_h(t_n)\approx \lambda(t_n)$.

The postprocessing step to compute the approximated scattered field can be
explained in a similar way. The $s$-domain semidiscrete single and
double layer potentials correspond to vector valued functions with
domain $\mathbb R^3\setminus\Gamma$
\begin{eqnarray*}
\mathbb S_h(s)_j &:=& \int_\Gamma \frac{e^{-s|\punto-\mathbf
y|}}{4\pi |\punto-\mathbf y|}N_j(\mathbf y)\mathrm d\Gamma(\mathbf
y),\\
\mathbb D_h(s)_k &:=& \int_\Gamma \frac{(\punto-\mathbf
y)\cdot\boldsymbol\nu(\mathbf y)}{4\pi |\punto-\mathbf y|^2}
(1+s|\punto-\mathbf y|)\frac{e^{-s|\punto-\mathbf
y|}}{|\punto-\mathbf y|}M_k(\mathbf y)\mathrm d\Gamma(\mathbf y).
\end{eqnarray*}
The CQ method uses the same strategy as in \eqref{eq:3.4} to produce
sequences of vector valued functions $\mathbb S_h^\kappa[n]$ and
$\mathbb D_h^\kappa[n]$, defined in $\mathbb R^3\setminus\Gamma$, and
uses them to construct the approximations
\[
u_h^\kappa[n]=-\sum_{m=0}^n \mathbb
S_h^\kappa[m]\boldsymbol\lambda_h^\kappa[n-m]-\sum_{m=0}^n \mathbb
D_h^\kappa[m]\boldsymbol\varphi_h^\kappa[n-m], \qquad n\ge 0.
\]
In the two-dimensional case the expressions for the fully discrete
method are very similar, using Hankel functions instead of
exponential expressions. For instance,
\begin{eqnarray*}
\mathbb V_h(s)_{i,j}&:=&\frac\imath4\int_\Gamma\int_\Gamma
H^{(1)}_0(\imath s |\mathbf x-\mathbf y|) N_i(\mathbf x)N_j(\mathbf
y)\mathrm d\Gamma(\mathbf x)\mathrm d\Gamma(\mathbf y),\\
\mathbb K_h(s)_{i,k} &:=& -\frac{s}4 \int_\Gamma\int_\Gamma
H^{(1)}_1(\imath s|\mathbf x-\mathbf y|)\frac{(\mathbf x-\mathbf
y)\cdot\boldsymbol\nu(\mathbf y)}{|\mathbf x-\mathbf
y|}\,N_i(\mathbf x)\,M_k(\mathbf y)\,\mathrm d\Gamma(\mathbf
x)\mathrm d\Gamma(\mathbf y).
\end{eqnarray*}

\begin{remark}
When data are not approximated, the indices $k$ have to be ignored.
Instead of having a matrix $\mathbb K_h(s)$ we have operators
$H^{1/2}(\Gamma) \to \mathbb R^J$ obtained by substituting the basis
function $M_k$ by a general element of $H^{1/2}(\Gamma)$. In this
way, the matrix-vector products $\mathbb K_h^\kappa [m]
\boldsymbol\varphi_h^\kappa [n-m]$ have to be substituted by the
action of operators $\mathbb K^\kappa[m]:H^{1/2}(\Gamma) \to
\mathbb R^J$ on elements
$\varphi^\kappa[n-m]=\varphi(t_{n-m})\in H^{1/2}(\Gamma)$.
Similar changes have to be applied to the double layer potential and
to the matrix $\mathbb I_h$ which is now substituted by an operator
$H^{1/2}(\Gamma) \to \mathbb R^J$ corresponding to testing a
function with the basis functions $N_i$.
\end{remark}

The analysis of the difference between the semidiscrete and the fully discrete solution at the different time-steps is carried out separately for the three time-discretization methods. This is done in Section \ref{sec:6}.

\section{The semidiscrete Galerkin projection}\label{sec:4}

Consider a casual smooth function $\lambda:\mathbb R \to H^{-1/2}(\Gamma)$ and a finite dimensional space $X_h \subset H^{-1/2}(\Gamma)$. The aim of this section is the analysis of the semidiscrete discretization process looking for a causal function $\lambda_h^{\mathrm G}:\mathbb R\to X_h$ such that
\begin{equation}\label{eq:4.1}
\langle \mu_h,(\mathcal V*\lambda_h^{\mathrm G})(t)\rangle_\Gamma=\langle \mu_h,(\mathcal V*\lambda)(t)\rangle_\Gamma \qquad \forall \mu_h\in X_h, \qquad \forall t
\end{equation}
and outputs the potential
\begin{equation}\label{eq:4.2}
u_h^{\mathrm G}:=\mathcal S*\lambda_h^{\mathrm G}.
\end{equation}
 Using a simple Laplace transform argument and the estimates of \cite{BamHaD86a} (see also \cite{LalSay09}), it is easy to prove that \eqref{eq:4.1} has at most one continuous causal solution.

Before proceeding to state and prove the main result of this section, we are going to introduce some constants related to the geometry of the problem and associated functional inequalities.

\subsection{Some inequalities}\label{sec:4.1}

Let  $R>0$ be such that
\[
\overline{\Omega_-} \subset B_0:=B(\mathbf 0;R)=\{ \mathbf x\in \mathbb R^d\,:\, |\mathbf x|< R\}
\]
and let us consider the balls $B_T:=B(\mathbf 0;R+T)$ for $T\ge 0$. Let then $C_T>0$ be taken so that
\begin{equation}\label{eq:4.3}
\| u\|_{B_T}\le C_T \| \nabla u\|_{B_T} \qquad \forall u\in H^1_0(B_T).
\end{equation}
A simple scaling argument shows that we can take $C_T=C_0(1+T/R)$. Therefore, {\em the constant $C_T$ grows linearly with $T$.} The trace operator on the boundary $\partial B_T$ will be denoted $\gamma_T$.

We will also consider a constant for the following two-sided trace inequality:
\begin{equation}\label{eq:4.4}
\|\gamma u\|_{1/2,\Gamma}\le C_\Gamma \| u\|_{1,B_0} \qquad \forall u\in H^1(B_0).
\end{equation}
Next, we consider a one-sided lifting of the trace onto $\Gamma$ in the form of a bounded linear operator $L:H^{1/2}(\Gamma) \to H^1(\mathbb R^d\setminus\Gamma)$ such that
\begin{equation}\label{eq:4.5}
L \varphi \equiv 0 \quad\mbox{in $\Omega_+$} \quad\mbox{and}\quad \gamma^-L\varphi=\varphi \quad \forall \varphi \in H^{1/2}(\Gamma).
\end{equation}
Note that $\jump{\gamma L\varphi}=\varphi$ and that there exists $C_L>0$ such that
\begin{equation}\label{eq:4.6}
\| L\varphi\|_{1,\mathbb R^d\setminus\Gamma} =\| L\varphi\|_{1,\Omega_-}\le C_L \|\varphi\|_{1/2,\Gamma} \qquad \forall \varphi \in H^{1/2}(\Gamma).
\end{equation}
Finally, using the weak definition of the normal derivative, we can fix a constant $C_\nu>0$ such that
\begin{equation}\label{eq:4.7}
\|\jump{\partial_\nu u}\|_{-1/2,\Gamma}\le C_\nu \Big( \| \Delta u\|_{B_0\setminus\Gamma}^2+\|\nabla u\|_{B_0\setminus\Gamma}^2\Big)^{1/2}\qquad \forall u\in H^1_\Delta (B_0\setminus\Gamma).
\end{equation}

\subsection{Estimates for the Galerkin projection}\label{sec:4.2}

\begin{theorem}\label{th:4.1}
Let $\lambda \in \mathcal W^2_0(\mathbb R;H^{-1/2}(\Gamma))$. Then the solution of the semidiscrete problem \eqref{eq:4.1} and its associated potential \eqref{eq:4.2} satisfy
\[
\lambda_h^{\mathrm G}\in \mathcal C_0(\mathbb R;H^{-1/2}(\Gamma)), \qquad u_h^{\mathrm G}\in \mathcal C^1_0(\mathbb R;H^1(\mathbb R^d)).
\]
Moreover, for all $t\ge 0$,
\begin{eqnarray}\label{eq:4.8}
\| u_h^{\mathrm G}(t)\|_{1,\mathbb R^d} & \le & 2 C_\Gamma \Big(\| \lambda(t)\|_{-1/2,\Gamma}+\sqrt{1+C_t^2}\,B_2^{-1/2}(\lambda,t)\Big),\\
\label{eq:4.9}
\| \lambda_h^{\mathrm G}(t)\|_{-1/2,\Gamma} & \le & (1+ C_\Gamma C_\nu )\| \lambda(t)\|_{-1/2,\Gamma}+ \sqrt2 C_\Gamma C_\nu B_2^{-1/2}(\lambda,t),
\end{eqnarray}
where
\[
B_2^{-1/2}(\lambda,t):=\int_0^t \Big(\|\lambda(\tau)\|_{-1/2,\Gamma}+\|\ddot\lambda(\tau)\|_{-1/2,\Gamma}\Big)\mathrm d\tau.
\]
\end{theorem}

The proof of Theorem \ref{th:4.1} will occupy the remainder of this section. The proof will never use that $X_h$ is finite dimensional. If we take $X_h=H^{-1/2}(\Gamma)$, \eqref{eq:4.8} gives a bound for the single layer acoustic operator that reproves  \cite[Theorem 3.1]{DomSaySB}.

Theorem \ref{th:4.1} will be proved for $\lambda \in \mathcal C^2_0(\mathbb R;H^{-1/2}(\Gamma))$. The extension to the general case follows by a simple density argument. Also, we will prove the results in a finite interval $[0,T]$, and bounds \eqref{eq:4.8}-\eqref{eq:4.9} will only be proved for $t=T$. Since $T$ is arbitrary, this is equivalent to having proved the results for any $t$.

Because of the finite speed of propagation of solutions to the wave equation, it is possible to understand (formally at the beginning) $u_h^{\mathrm G}$ as a solution of the following  wave propagation problem on a truncated domain with non-standard transmission conditions (see \cite{LalSay09, SaySB})
\begin{subequations}\label{eq:4.10}
\begin{alignat}{4}
u_h^{\mathrm G}(t) \in H^1_0(B_T) & \qquad & & 0 \le t\le T,\\
\ddot u_h^{\mathrm G}(t)=\Delta u_h^{\mathrm G}(t) & \qquad & & 0 \le t\le T,\\
\gamma u_h^{\mathrm G}(t)-(\mathcal V*\lambda)(t)\in X_h^\circ& \qquad & & 0 \le t\le T,\\
\jump{\partial_\nu u_h^{\mathrm G}(t)}\in X_h& \qquad & & 0 \le t\le T,\\
u_h^{\mathrm G}(0)=\dot u_h^{\mathrm G}(0)=0.
\end{alignat}
\end{subequations}
The set $X_h^\circ$ is the polar set or annihilator of $X_h$, i.e.,
\[
X_h^\circ:=\{\rho\in H^{1/2}(\Gamma)\,:\, \langle \mu_h,\rho\rangle_\Gamma =0 \quad\forall \mu_h \in X_h\}.
\]
If $u\in H^1_\Delta(B_T\setminus\Gamma)$, the condition $\jump{\partial_\nu u}\in X_h$ can be rewritten as a set of restrictions 
\[
\langle \jump{\partial_\nu u},\rho\rangle_\Gamma =0\qquad\forall \rho \in X_h^\circ,
\]
or equivalently, as
\begin{equation}\label{eq:4.11}
(\nabla u,\nabla v)_{B_T\setminus\Gamma}+(\Delta u,v)_{B_T\setminus\Gamma}=0 \qquad \forall v \in V_h^T
\end{equation}
where
\begin{equation}\label{eq:4.12}
V_h^T :=\{ v \in H^1_0(B_T)\,:\, \gamma v\in X_h^\circ\}.
\end{equation}

\begin{proposition}[Uniqueness]\label{prop:4.1b}
Problem \eqref{eq:4.10} has at most one solution
\begin{equation}
\label{eq:4.30}
u_h^{\mathrm G}\in \mathcal C^2([0,T];L^2(B_T))\cap \mathcal C^1([0,T];H^1_0(B_T))\cap \mathcal C([0,T];H^1_\Delta(B_T\setminus\Gamma)).
\end{equation}
If this solution exists and $\lambda_h^{\mathrm G}:=\jump{\partial_\nu u_h^{\mathrm G}}$, then $u_h^{\mathrm G}$ and $\lambda_h^{\mathrm G}$ coincide with the solution of \eqref{eq:4.1}-\eqref{eq:4.2} on the interval $[0,T]$.
\end{proposition}

\begin{proof} Uniqueness of solution follows from an elementary energy argument. Careful, but not complicated, use of extension operators from $B_T$ to $\mathbb R^d$ and from $[0,\infty)$ to $\mathbb R$ in the sense of vector valued distributions, can be used following 
\cite[Sections 4 \& 5]{DomSaySB}, to prove the relation between a strong solution of the transmission problem and a weak distributional solution of \eqref{eq:4.1}-\eqref{eq:4.2}.
\end{proof} 

\subsection{An underlying initial value problem}\label{sec:4.3}

Associated to the space $V_h^T$ given in \eqref{eq:4.12}, we consider the space
\begin{eqnarray}\label{eq:4.13}
D_h^T &:=& \{ v \in V_h^T \,:\, \Delta v \in L^2(B_T\setminus\Gamma), \jump{\partial_\nu v}\in X_h\}\\
\nonumber
&=& \{ v \in V_h^T \cap H^1_\Delta (B_T\setminus\Gamma)\,:\, (\nabla v,\nabla w)_{B_T}+(\Delta v,w)_{B_T\setminus\Gamma}=0\quad\forall w\in V_h^T\}
\end{eqnarray}
(both definitions coincide by \eqref{eq:4.11}). In the frame of the triple $D_h^T \subset V_h^T\subset L^2(B_T)$ we can consider the unbounded operator $\Delta:D_h^T\subset L^2(B_T) \to L^2(B_T)$ (the distributional Laplacian in $B_T\setminus\Gamma)$) and the initial value problem
\begin{equation}\label{eq:4.14}
\ddot v(t)=\Delta v(t)+f(t) \quad \forall t\ge 0, \qquad v(0)=\dot v(0)=0.
\end{equation}
{\em Strong solutions} of this problem will be those with values in $D_h^T$. Following \cite[Section 8 and Appendix]{DomSaySB}, we can also consider {\em weak solutions} of \eqref{eq:4.14}. In order to do that, we start in the Gelfand triple $V_h^T \subset L^2(B_T)\cong L^2(B_T)' \subset (V_h^T)'$, we extend $\Delta$ to an unbounded operator $\Delta: V_h^ T \subset (V_h^T)' \to (V_h^T)'$ and we only look for solutions with values in $V_h^T$. In this case, the additional boundary condition that appears in the definition of $D_h^T$ and the extended definition of the Laplace operator are part of the same expression: namely, the differential equation \eqref{eq:4.14} is to be understood as
\begin{equation}\label{eq:4.15}
(\ddot v(t),w)_{(V_h^T)'\times V_h^T}+(\nabla v(t),\nabla w)_{B_T}=(f(t),w)_{B_T} \qquad \forall w \in V_h^T.
\end{equation}
The following two results follow from \cite[Appendix]{DomSaySB}.

\begin{proposition}[Strong solutions]\label{prop:4.2}
If $f:[0,\infty)\to V_h^T$ is continuous, then there exists $v:[0,\infty)\to D_h^T$ such that
\begin{equation}\label{eq:4.15a}
v \in \mathcal C^2([0,\infty);L^2(B_T)) \cap \mathcal C^1([0,\infty);H^1_0(B_T))\cap \mathcal C([0,\infty);H^1_\Delta(B_T\setminus\Gamma)),
\end{equation}
satisfying \eqref{eq:4.14} and the bounds
\begin{eqnarray*}
\| v(t)\|_{B_T} & \le & C_T \int_0^ t \| f(\tau)\|_{B_T}\mathrm d\tau,\\
\|\nabla v(t)\|_{B_T} & \le &  \int_0^ t \| f(\tau)\|_{B_T}\mathrm d\tau,\\
\| \Delta v(t)\|_{B_T\setminus\Gamma} & \le & \int_0^ t \| \nabla f(\tau)\|_{B_T}\mathrm d\tau.
\end{eqnarray*}
\end{proposition}

\begin{proposition}[Weak solutions]\label{prop:4.3}
If $f:[0,\infty)\to L^2(B_T)$ is continuous, then there exists $v:[0,\infty) \to V_h^T$ such that
\[
v \in \mathcal C^2([0,\infty);(V_h^T)')\cap \mathcal C^1([0,\infty);L^2(B_T))\cap \mathcal C([0,\infty);H^1_0(B_T)),
\]
satisfying the weak form of \eqref{eq:4.14} (see \eqref{eq:4.15}) and the bounds
\[
C_T^{-1}\|v(t)\|_{B_T}\le \| \nabla v(t)\|_{B_T}\le \int_0^ t \| f(\tau)\|_{B_T}\mathrm d\tau. 
\]
Moreover, $\widetilde v(t):=\int_0^t v(\tau)\mathrm d\tau$ is a continuous $D_h^T$-valued function.
\end{proposition}

\subsection{A decomposition of $u_h^{\mathrm G}$}

As explained in Section \ref{sec:4.2}, we are going to deal with $\lambda\in \mathcal C^2([0,T];H^{-1/2}(\Gamma))$ such that $\lambda(0)=\dot\lambda(0)=0$ and obtain bounds in the interval $[0,T]$ as well.
The proof of this result follows from the decomposition
\begin{equation}\label{eq:4.15b}
u_h^{\mathrm G}=\mathcal S*\lambda + u_h^0+v_h^0,
\end{equation}
where $u_h^0:[0,\infty)\to V_h^T$ solves the steady-state transmission problems (for all $t\ge 0$)
\begin{equation}\label{eq:4.16}
u_h^ 0(t)\in V_h^T, \qquad -\Delta u_h^0(t)+u_h^0(t)=0 \quad\mbox{in $B_T\setminus\Gamma$},\qquad \jump{\partial_\nu u_h^0(t)}+\lambda(t)\in X_h, 
\end{equation}
and $v_h^0:[0,\infty) \to D_h^T$ is a solution of the evolution problem
\begin{equation}\label{eq:4.17}
\begin{array}{l}
\ddot v_h^0(t)=\Delta v_h^0(t)+u_h^0(t)-\ddot u_h^0(t)=\Delta v_h^0(t)+\Delta u_h^0(t)-\ddot u_h^0(t)\qquad t\ge 0,\\[1.5ex]
v_h^0(0)=\dot v_h^0(0)=0.
\end{array}
\end{equation}
We start by analyzing the three terms in \eqref{eq:4.15b} one by one. Note that we still need to show that the decomposition \eqref{eq:4.15b} holds true, that is, that the sum of the three functions in the right hand side of \eqref{eq:4.15b} is $u_h^{\mathrm G}$.

\paragraph{1.}  By Theorem 3.1 and Proposition 4.2 in \cite{DomSaySB}, it follows that
\begin{equation}\label{eq:4.18}
\mathcal S*\lambda \in \mathcal C^2([0,T];L^2(B_T))\cap \mathcal C^1([0,T];H^1_0(B_T))\cap \mathcal C([0,T];H^1_\Delta(B_T\setminus\Gamma)),
\end{equation}
that for all $t\in [0,T]$,
\begin{equation}\label{eq:4.18b}
\begin{array}{l}
\frac{\mathrm d^2}{\mathrm d t^2}(\mathcal S*\lambda)(t)=\Delta (\mathcal S*\lambda)(t), \qquad \jump{\partial_\nu(\mathcal S*\lambda)(t)}=\lambda(t),\\[1.5ex]
(\mathcal S*\lambda)(0)=\frac{\mathrm d}{\mathrm d t}(\mathcal S*\lambda)(0)=0,
\end{array}
\end{equation}
and
\begin{equation}\label{eq:4.A}
\| (\mathcal S*\lambda)(t)\|_{1,B_T}\le C_\Gamma \Big(\|\lambda(t)\|_{-1/2,\Gamma}+\sqrt{1+C_t^2}\, B_2^{-1/2}(\lambda,t)\Big) \qquad 0\le t\le T.
\end{equation}
\paragraph{2.} We next analyze the behavior of $u_h^0$. The variational formulation of \eqref{eq:4.16} is
\begin{equation}\label{eq:4.19}
u_h^0(t)\in V_h^T,\qquad (\nabla u_h^0(t),\nabla v)_{B_T}+(u_h^0(t),v)_{B_T}=-\langle \lambda(t),\gamma v\rangle_\Gamma\quad \forall v\in V_h^T.
\end{equation}
This is a well posed problem, that depends on $t$, only because data depend on $t$. In particular, 
\begin{equation}\label{eq:4.20}
u_h^0\in \mathcal C^2([0,T];H^1_0(B_T)\cap H^1_\Delta(B_T\setminus\Gamma)), \qquad u_h^0(0)=\dot u_h^0(0)=0.
\end{equation}
The variational formulation \eqref{eq:4.19} and the trace inequality \eqref{eq:4.6} show that
\begin{equation}\label{eq:4.B}
\| u_h^0(t)\|_{1,B_T} \le C_\Gamma\| \lambda(t)\|_{-1/2,\Gamma}.
\end{equation}
On the other hand, since $\Delta u_h^0(t)=u_h^0(t)$, then \eqref{eq:4.7} implies that
\begin{equation}\label{eq:4.C}
\|\jump{\partial_\nu u_h^0(t)}\|_{-1/2,\Gamma}\le C_\nu \| u_h^0(t)\|_{1,B_T}\le C_\nu C_\Gamma \| \lambda(t)\|_{-1/2,\Gamma}.
\end{equation}
Differentiating \eqref{eq:4.A} twice with respect to $t$, it also follows that
\begin{equation}\label{eq:4.D}
\| \ddot u_h^0(t)\|_{1,B_T} \le C_\Gamma\| \ddot \lambda(t)\|_{-1/2,\Gamma}.
\end{equation}
\paragraph{3. } We can apply Proposition \ref{prop:4.2} to the solution of \eqref{eq:4.17}, taking $f:=u_h^0-\ddot u_h^0$. 
It then follows that
\begin{equation}\label{eq:4.W}
v_h^0 \in \mathcal C^2([0,\infty);L^2(B_T)) \cap \mathcal C^1([0,\infty);H^1_0(B_T))\cap \mathcal C([0,
\infty);H^1_\Delta(B_T\setminus\Gamma)).
\end{equation}
We also obtain the bounds
\begin{equation}\label{eq:4.E}
\| v_h^0(t)\|_{B_T}\le C_t \int_0^ t \| u_h^0(\tau)-\ddot u_h^0(\tau)\|_{B_T}\mathrm d\tau \le C_TC_\Gamma B_2^{-1/2}(\lambda,t)
\end{equation}
(we have used \eqref{eq:4.B} and \eqref{eq:4.D} in the last step),
\begin{equation}\label{eq:4.F}
\| \nabla v_h^0(t)\|_{B_T}\le C_\Gamma  B_2^{-1/2}(\lambda,t)
\end{equation}
and 
\begin{equation}\label{eq:4.G}
\| \Delta v_h^0(t)\|_{B_T\setminus\Gamma}\le C_\Gamma  B_2^{-1/2}(\lambda,t).
\end{equation}
From the bound for the jump of the normal derivative \eqref{eq:4.7} and \eqref{eq:4.F}-\eqref{eq:4.G}, it follows that
\begin{equation}\label{eq:4.H}
\| \jump{\partial_\nu v_h^0(t)}\|_{-1/2,\Gamma}\le \sqrt2 C_\Gamma C_\nu B_2^{-1/2}(\lambda,t).
\end{equation}

\subsection{Proof of Theorem \ref{th:4.1}}

Let now $u_h^{\mathrm G}$ be defined by \eqref{eq:4.15b}. To prove that $u_h^{\mathrm G}$ satisfies \eqref{eq:4.30}, we just have to use  \eqref{eq:4.18}, \eqref{eq:4.20} and \eqref{eq:4.W}. To prove that  $u_h^{\mathrm G}$ satisfies problem \eqref{eq:4.10}, we add the equations that are satisfied by the three components of the sum, namely \eqref{eq:4.18b}, \eqref{eq:4.16} and \eqref{eq:4.17}. The uniqueness result of Proposition \ref{prop:4.1b} shows then that the function defined by \eqref{eq:4.15b} can be identified with the solution of \eqref{eq:4.1}-\eqref{eq:4.2} in the interval $[0,T]$. 

By \eqref{eq:4.A}, \eqref{eq:4.B}, \eqref{eq:4.E} and \eqref{eq:4.F}, it follows that
\begin{equation}\label{eq:4.31}
\| u_h^{\mathrm G}(t)\|_{1,B_T} \le  2 C_\Gamma \|\lambda(t)\|_{-1/2,\Gamma}+ C_\Gamma\Big(\sqrt{1+C_t^2}+\sqrt{1+C_T^2}\Big)B_2^{-1/2}(\lambda,t).
\end{equation}
Taking $t=T$ and using that $T$ is arbitrary, \eqref{eq:4.8} follows. 

Since $\lambda_h^{\mathrm G}(t)=\jump{\partial_\nu u_h^{\mathrm G}(t)}=\lambda(t)+\jump{\partial_\nu u_h^0(t)}+\jump{\partial_\nu v_h^0(t)}$, and $\jump{\partial_\nu\cdot}:H^1_\Delta(B_T\setminus\Gamma)\to H^{-1/2}(\Gamma)$ is bounded, then \eqref{eq:4.20} and \eqref{eq:4.W} imply that $\lambda_h^{\mathrm G}\in \mathcal C([0,T],H^{-1/2}(\Gamma))$. The uniqueness argument of Proposition \ref{prop:4.1b} then proves that the solution of \eqref{eq:4.1} is a causal continuous $H^{-1/2}(\Gamma)$-valued function. Finally, inequalities \eqref{eq:4.C} and \eqref{eq:4.H} prove \eqref{eq:4.9}. 

\section{The semidiscrete Galerkin solver}\label{sec:5}

In this section we study how Galerkin semidiscretization depends on data. Our starting point is a causal function $\varphi:\mathbb R\to H^{1/2}(\Gamma)$. We then consider the function $\lambda_h^\varphi:\mathbb R\to X_h$ such that
\begin{equation}\label{eq:5.1}
\langle \mu_h,(\mathcal V*\lambda_h^\varphi)(t)\rangle_\Gamma=\langle \mu_h,\smallfrac12 \varphi(t)-(\mathcal K*\varphi)(t)\rangle_\Gamma\qquad \forall \mu_h \in X_h
\end{equation}
and the associated exterior solution
\begin{equation}\label{eq:5.2}
u_h^\varphi:=-\mathcal S*\lambda_h^\varphi-\mathcal D*\varphi.
\end{equation}
 Using a simple Laplace transform argument and the estimates of \cite{BamHaD86a} (see also \cite{LalSay09}), it is easy to prove that \eqref{eq:5.1} has at most one continuous causal solution. Moreover, uniqueness can be also established for weaker solutions, where for instance, $\lambda_h^\varphi$ is the distributional derivative of a continuous causal $X_h$-valued function.

\subsection{Estimates for the Galerkin solver}

\begin{theorem}\label{th:5.1}
Let $\varphi\in \mathcal W^4_0(\mathbb R; H^{1/2}(\Gamma))$. Then the solution of the semidiscrete problem \eqref{eq:5.1} and its associated potential \eqref{eq:5.2} satisfy
\[
\lambda_h^{\varphi}\in \mathcal C(\mathbb R;H^{-1/2}(\Gamma)), \qquad u_h^{\varphi}\in \mathcal C^1(\mathbb R;H^1(\mathbb R^d\setminus\Gamma)).
\]
Moreover, for all $t\ge 0$,
\begin{eqnarray}\label{eq:5.5}
\| u_h^\varphi(t)\|_{1,\mathbb R^d\setminus\Gamma} & \le & C_L \Big(\| \varphi(t)\|_{1/2,\Gamma}+\sqrt{1+C_t^2}\,B_2^{1/2}(\varphi,t)\Big),\\
\label{eq:5.6}
\| \lambda_h^\varphi(t)\|_{-1/2,\Gamma} & \le & \sqrt2 C_\nu C_L \Big(4\|\varphi(t)\|_{1/2,\Gamma}+2\|\ddot\varphi(t)\|_{1/2,\Gamma}+ B_4^{1/2}(\varphi,t)\Big),
\end{eqnarray}
where
\begin{eqnarray*}
B_2^{1/2}(\varphi,t)&:=&\int_0^t \Big(\|\varphi(\tau)\|_{1/2,\Gamma}+\|\ddot\varphi(\tau)\|_{1/2,\Gamma}\Big)\mathrm d\tau,\\
B_4^{1/2}(\varphi,t)&:=& 4B_2^{1/2}(\varphi,t)+B_2^{1/2}(\ddot\varphi,t).
\end{eqnarray*}
\end{theorem}

The proof of this result follows partially the steps of the proof of Theorem \ref{th:4.1}. The analysis will be more involved because of the occurrence of a non-homogeneous essential transmission condition (see \eqref{eq:5.3d} below), that cannot be easily lifted with a continuous potential. Like in Section \ref{sec:4}, we will prove the estimates on a fixed time interval $[0,T]$, taking advantage of the fact that finite speed of propagation will allow us to impose a homogeneous Dirichlet boundary condition in $\partial B_T$. Again, we will only assume that $\varphi\in \mathcal C^4([0,T];H^{1/2}(\Gamma))$ with $\varphi^{(k)}(0)=0$ for $k\le 3$. The result for general $\varphi$ can be extended with a density argument.

One of the keys towards the proof of the result lies in the fact if $u_h^\varphi$ is smooth enough, then $u_h^\varphi$ is a strong solution of the following problem:
\begin{subequations}\label{eq:5.3}
\begin{alignat}{4}
\label{eq:5.3a}
u_h^{\varphi}(t) \in H^1(B_T\setminus\Gamma) & \qquad & & 0 \le t\le T,\\
\label{eq:5.3b}
\gamma_Tu_h^{\varphi}(t)=0 & & & 0\le t\le T,\\
\label{eq:5.3c}
\ddot u_h^{\varphi}(t)=\Delta u_h^{\mathrm G}(t) & \qquad & & 0 \le t\le T,\\
\label{eq:5.3d}
\jump{\gamma u_h^{\varphi}(t)}=\varphi(t) & \qquad & & 0 \le t\le T,\\
\label{eq:5.3e}
\gamma^+ u_h^\varphi(t)\in X_h^\circ  & \qquad & & 0 \le t\le T,\\
\label{eq:5.3f}
\jump{\partial_\nu u_h^{\varphi}(t)}\in X_h& \qquad & & 0 \le t\le T,\\
\label{eq:5.3g}
u_h^{\varphi}(0)=\dot u_h^{\varphi}(0)=0.
\end{alignat}
\end{subequations}
At the same time, we can combine \eqref{eq:5.3c} and \eqref{eq:5.3f}, multiply by $w\in V_h^T$ and integrate, to obtain a weaker form of the differential equation and the natural boundary condition
\begin{equation}\label{eq:5.4}
(\ddot u_h^\varphi(t),w)_{B_T}+(\nabla u_h^\varphi(t),\nabla w)_{B_T\setminus\Gamma}=0 \qquad \forall w\in V_h^T, \qquad 0\le t\le T.
\end{equation}
If we now define
\begin{equation}\label{eq:5.4b}
w_h^\varphi(t):=\int_0^ t u_h^\varphi(\tau)\mathrm d\tau,
\end{equation}
it also follows that
\begin{equation}\label{eq:5.4c}
(\ddot w_h^\varphi(t),w)_{B_T}+(\nabla w_h^\varphi(t),\nabla w)_{B_T\setminus\Gamma}=0 \qquad \forall w\in V_h^T, \qquad 0\le t\le T.
\end{equation}
The following double uniqueness result will help us recognize $u_h^\varphi$ in the two decompositions that will be given below.

\begin{proposition}[Uniqueness]\label{prop:5.2} 
 There exists at most one 
\[
u_h^\varphi \in \mathcal C^1([0,T];L^2(B_T))\cap \mathcal C([0,T];H^1(B_T\setminus\Gamma))
\]
satisfying the essential boundary and transmission conditions \eqref{eq:5.3b}, \eqref{eq:5.3d}, and \eqref{eq:5.3e}, the initial conditions \eqref{eq:5.3g}, and such that $w_h^\varphi$ defined with \eqref{eq:5.4b} satisfies \eqref{eq:5.4c}. If such a solution exists and, in addition, $w_h^\varphi\in \mathcal C([0,T];H^1_\Delta(B_T\setminus\Gamma)$, then $\lambda_h^\varphi=-\frac{\mathrm d}{\mathrm dt}\jump{\partial_\nu w_h^\varphi}$ (with differentiation in the sense of $H^{-1/2}(\Gamma)$-valued distributions) solves \eqref{eq:5.1} and $u_h^\varphi$ satisfies \eqref{eq:5.2}.
\end{proposition}

\begin{proof}  Let $u_h^{\varphi=0}$ be a solution of \eqref{eq:5.3b},\eqref{eq:5.3d} with $\varphi=0$, \eqref{eq:5.3e} and \eqref{eq:5.3g} such that the corresponding $w_h^{\varphi=0}$ satisfies \eqref{eq:5.4c}. Then $u_h^{\varphi=0}:[0,T]\to V_h^T$ is continuous and we have enough regularity to test \eqref{eq:5.4c} with $u_h^{\varphi=0}(t)=\dot w_h^{\varphi=0}(t)\in V_h^T$ and prove that
\[
\frac{\mathrm d}{\mathrm d t}\Big( \| \dot w_h^{\varphi=0}(t)\|_{B_T}^2+ \| \nabla w_h^{\varphi=0}(t)\|_{B_T\setminus\Gamma}^2\Big)=0.
\]
The proof of uniqueness of solution is now straightforward. The final statement follows from the same arguments developed in \cite[Section 5]{DomSaySB}. 
\end{proof}

\begin{corollary}\label{prop:5.3}
There exists at most one 
\[
u_h^\varphi \in \mathcal C^2([0,T];L^2(B_T))\cap \mathcal C^1([0,T];H^1(B_T\setminus\Gamma))\cap \mathcal C([0,T];H^1_\Delta(B_T\setminus\Gamma))
\]
that solves \eqref{eq:5.3}. Moreover, if such a solution exists, then $\lambda_h^\varphi:=-\jump{\partial_\nu u_h^\varphi}$ solves  \eqref{eq:5.1} and  $u_h^\varphi$ satisfies \eqref{eq:5.2}.
\end{corollary}

\subsection{A first decomposition of $u_h^\varphi$}

For the arguments of this section, we only need $\varphi\in \mathcal C^2([0,T];H^{1/2}(\Gamma))$ with $\varphi(0)=\dot\varphi(0)=0$.
In a first step, we formally decompose
\begin{equation}
\label{eq:5.7}
u_h^\varphi= u_h^0+v_h^0,
\end{equation}
where $u_h^0:[0,T]\to H^1(B_T\setminus\Gamma)$ is the solution of the variational problems (for $0\le t\le T$)
\begin{equation}\label{eq:5.8}
\begin{array}{l}
u_h^0(t)\in H^1(B_T\setminus\Gamma),\\[1.5ex]
\gamma_T u_h^0(t)=0, \quad \jump{\gamma u_h^0(t)}=\varphi(t), \qquad \gamma^+ u_h^0(t)\in X_h^\circ,\\[1.5ex]
(\nabla u_h^0(t),\nabla v)_{B_T\setminus\Gamma}+(u_h^0(t),v)_{B_T}=0\qquad\forall v \in V_h^T,
\end{array}
\end{equation}
and $v_h^0: [0,T]\to V_h^T$ is a solution of
\begin{equation}\label{eq:5.9}
\begin{array}{l}
(\ddot v_h^0(t),w)_{(V_h^T)'\times V_h^T}+(\nabla v_h^0(t),\nabla w)_{B_T}\\
\hspace{4cm}=(u_h^0(t)-\ddot u_h^0(t),w)_{B_T} \quad \forall w\in V_h^T, \quad 0\le t\le T,\\[1.5ex]
v_h^0(0)=\dot v_h^0(0)=0.
\end{array}
\end{equation}
\paragraph{1.} Problem \eqref{eq:5.8} has a unique solution by a simple coercivity argument. Since dependence on $t$ happens only through the non-homogeneous essential transmission condition (compare with \eqref{eq:4.19}, where the condition was natural), it is simple to prove that
\begin{equation}\label{eq:5.10}
u_h^0\in \mathcal C^2([0,T];H^1(B_T\setminus\Gamma)), \qquad u_h^0(0)=\dot u_h^0(0)=0.
\end{equation}
Using the lifting operator $L$ given in \eqref{eq:4.5}-\eqref{eq:4.6} and taking $u_h^0(t)-L\varphi(t)\in V_h^T$ as a test function in \eqref{eq:5.8}, we can prove that
\begin{equation}\label{eq:5.A}
\|u_h^0(t)\|_{1,B_T\setminus\Gamma} \le C_L \|\varphi(t)\|_{1/2,\Gamma}.
\end{equation}
Taking second derivatives with respect to time in \eqref{eq:5.8} and the using same kind of argument, we prove
\begin{equation}\label{eq:5.B}
\|\ddot u_h^0(t)\|_{1,B_T\setminus\Gamma} \le C_L \|\ddot \varphi(t)\|_{1/2,\Gamma}.
\end{equation}
\paragraph{2.} Proposition \ref{prop:4.3} can now be invoked to prove that the evolution problem \eqref{eq:5.9} has a unique solution
\begin{equation}\label{eq:5.11}
v_h^0\in \mathcal C^1([0,T];H^1_0(B_T))\cap \mathcal C([0,T];L^2(B_T)),
\end{equation}
satisfying the bounds
\begin{equation}\label{eq:5.C}
C_T^{-1}\| v_h^0(t)\|_{B_T}\le \|\nabla v_h^0(t)\|_{B_T}\le \int_0^t \| u_h^0(\tau)-\ddot u_h^0(\tau)\|_{B_T}\mathrm d\tau\le C_L B_2^{1/2}(\varphi,t),
\end{equation}
where we have used \eqref{eq:5.A}-\eqref{eq:5.B} in the last inequality.

\subsection{A second decomposition of $u_h^\varphi$}

An alternative decomposition to \eqref{eq:5.7} is needed to bound the density $\lambda_h^\varphi$. From this moment on, we need $\varphi\in \mathcal C^4_0([0,T];H^{1/2}(\Gamma))$. We now write
\begin{equation}
\label{eq:5.20}
u_h^\varphi= u_h^1+v_h^1,
\end{equation}
where $u_h^1:[0,T]\to H^1(B_T\setminus\Gamma)$ is the solution of the variational problems (for $0\le t\le T$)
\begin{equation}\label{eq:5.21}
\begin{array}{l}
u_h^1(t)\in H^1(B_T\setminus\Gamma),\\[1.5ex]
\gamma_T u_h^1(t)=0, \quad \jump{\gamma u_h^1(t)}=\varphi(t), \qquad \gamma^+ u_h^1(t)\in X_h^\circ,\\[1.5ex]
(\nabla u_h^1(t),\nabla v)_{B_T\setminus\Gamma}+(u_h^1(t),v)_{B_T}=(L\ddot\varphi(t)-L\varphi(t),v)_{B_T}\qquad\forall v \in V_h^T,
\end{array}
\end{equation}
and $v_h^1: [0,T]\to D_h^T$ is a solution of
\begin{equation}\label{eq:5.22}
\begin{array}{l}
\ddot v_h^1(t)=\Delta v_h^1(t)+f(t),\quad 0\le t\le T,\\[1.5ex]
v_h^1(0)=\dot v_h^1(0)=0,
\end{array}
\end{equation}
where $f:=u_h^1-\ddot u_h^1 + L\varphi-L\ddot\varphi=\Delta u_h^1-\ddot u_h^1:[0,T]\to V_h^T$
is continuous. 

\paragraph{1.} It is clear that $u_h^1\in \mathcal C^2([0,T];H^1_\Delta(B_T\setminus\Gamma))$ and $u_h^1(0)=\dot u_h^1(0)=0$. Also, using $v=u_h^1(t)-L\varphi(t)\in V_h^T$ as test function in \eqref{eq:5.21}, we can prove that
\begin{eqnarray}\label{eq:5.23}
\| u_h^1(t)\|_{1,B_T\setminus\Gamma}&\le& \|L\varphi(t)\|_{1,B_T\setminus\Gamma}+\| u_h^1(t)-L\varphi(t)\|_{1,B_T\setminus\Gamma}\\
&\le& C_L ( 3\|\varphi(t)\|_{1/2,\Gamma}+\|\ddot \varphi(t)\|_{1/2,\Gamma})\nonumber
\end{eqnarray}
and therefore
\begin{equation}\label{eq:5.24}
\|\Delta u_h^1(t)\|_{B_T\setminus\Gamma}\le C_L(4\|\varphi(t)\|_{1/2,\Gamma}+2\|\varphi(t)\|_{1/2,\Gamma})
\end{equation}
\paragraph{2.} 
Regularity of the solution of \eqref{eq:5.22} is given by \eqref{eq:4.15a} in Proposition \ref{prop:4.2}. Since \eqref{eq:5.23} can be differentiated twice with respect to time, we can bound
\begin{eqnarray*}
\|\nabla f(t)\|_{B_T}&\le & \| u_h^1(t)\|_{1,B_T\setminus\Gamma}+\| \ddot u_h^1(t)\|_{1,B_T\setminus\Gamma}+C_L(\|\varphi(t)\|_{1/2,\Gamma}+\|\ddot\varphi(t)\|_{1/2,\Gamma})\\
&\le & C_L ( 4\|\varphi(t)\|_{1/2,\Gamma}+ 5 \|\ddot \varphi(t)\|_{1/2,\Gamma}+\|\varphi^{(4)}(t)\|_{1/2,\Gamma})
\end{eqnarray*}
and thus
\begin{equation}\label{eq:5.25}
\|\Delta v_h^1\|_{B_T\setminus\Gamma} \le \int_0^t\|\nabla f(\tau)\|_{B_T}\mathrm d\tau\le C_L B_4(\varphi,t)\qquad 0\le t\le T.
\end{equation}

\subsection{Proof of Theorem \ref{th:5.1}}

We first define $u_h^\varphi$ with \eqref{eq:5.7}, where $u_h^0$ solves \eqref{eq:5.8} and $v_h^0$ solves \eqref{eq:5.9}. By \eqref{eq:5.8}, \eqref{eq:5.9}, \eqref{eq:5.10} and \eqref{eq:5.11}, it follows that $u_h^\varphi$ is the only weak solution of \eqref{eq:5.3} in the sense of Proposition \ref{prop:5.2} and it can be thus identified with the solution of \eqref{eq:5.1}-\eqref{eq:5.2}. As a direct consequence of \eqref{eq:5.A} and \eqref{eq:5.C}, it follows that
\begin{equation}\label{eq:5.D}
\| u_h^\varphi(t)\|_{1,B_T\setminus\Gamma}\le C_L \Big(\|\varphi(t)\|_{1/2,\Gamma}+\sqrt{1+C_T^2}\,B_2^{1/2}(\varphi,t)\Big), \qquad 0\le t\le T,
\end{equation}
and
\begin{equation}\label{eq:5.E}
\| \nabla u_h^\varphi(t)\|_{B_T}\le C_L  \Big(\|\varphi(t)\|_{1/2,\Gamma}+B_2^{1/2}(\varphi,t)\Big), \qquad 0\le t\le T.
\end{equation}

We next define $u_h^\varphi=u_h^1+v_h^1$, where $u_h^1$ satisfies \eqref{eq:5.21} and $v_h^1$ satisfies \eqref{eq:5.22}. It is clear that $u_h^\varphi$ is a strong solution of \eqref{eq:4.3} and therefore (Corollary \ref{prop:5.3}) coincides with the solution of \eqref{eq:5.1}-\eqref{eq:5.2}.  Since $u_h^\varphi\in \mathcal C([0,T];H^1_\Delta(B_T\setminus\Gamma)$, then $\lambda_h^\varphi =-\jump{\partial_\nu u_h^\varphi}\in \mathcal C([0,T];H^{-1/2}(\Gamma))$. By \eqref{eq:5.23} and \eqref{eq:5.24} it follows that
\begin{equation}\label{eq:5.26}
\| \Delta u_h^\varphi(t)\|_{B_T\setminus\Gamma}\le C_L \Big(4\|\varphi(t)\|_{1/2,\Gamma}+2\|\varphi(t)\|_{1/2,\Gamma}+ B_4^{1/2}(\varphi,t)\Big)\qquad 0\le t\le T.
\end{equation}
This inequality, \eqref{eq:5.E}, and \eqref{eq:4.7} prove finally that
\[
\|\lambda_h^\varphi(t)\|_{-1/2,\Gamma}\le \sqrt C_L C_\nu  \Big(4\|\varphi(t)\|_{1/2,\Gamma}+2\|\varphi(t)\|_{1/2,\Gamma} + B_4^{1/2}(\varphi,t)\Big)\qquad 0\le t\le T.
\]

\subsection{Proof of Theorem \ref{th:3.1}}

We use the notation \eqref{eq:4.1} for the Galerkin projection and \eqref{eq:5.1} for the Galerkin solver. Since, $(\Pi_h\lambda)_h^{\mathrm{G}}=\Pi_h \lambda$, we can bound
\[
\|(\lambda-\lambda_h)(t)\|_{-1/2,\Gamma} \le \| (\lambda-\Pi_h\lambda)(t)\|_{-1/2,\Gamma}+ \| (\lambda-\Pi_h\lambda)_h^{\mathrm G}(t)\|_{-1/2,\Gamma}+\|\lambda_h^{\varphi-\varphi_h}(t)\|_{-1/2,\Gamma}.
\]
The bound \eqref{eq:3.A} follows then from Theorems \ref{th:4.1} and \ref{th:5.1}. Similarly we write
\[
\| (u-u_h)(t)\|_{1,\mathbb R^d} \le \| \mathcal S*(\lambda-\Pi_h\lambda)(t)\|_{1,\mathbb R^d}+\| \mathcal S * (\lambda-\Pi_h\lambda)_h^{\mathrm G}(t)\|_{1,\mathbb R^d}+\|u_h^{\varphi-\varphi_h}(t)\|_{1,\mathbb R^d}
\]
and use \cite[Theorem 3.1]{DomSaySB}, Theorem \ref{th:4.1} and Theorem \ref{th:5.1} to prove \eqref{eq:3.B}.

\section{Analysis of time discretization}\label{sec:6}

\subsection{A passage to the Laplace domain}

For any value $s\in \mathbb C_+:=\{ s\in \mathbb C\,:\,\mathrm{Re }\, s >0\},$
we consider the fundamental solution of the differential operator $\Delta -s^2$, namely,
\[
\Phi(\mathbf x,\mathbf y;s):=\left\{
\begin{array}{ll}
\frac\imath4 H^{(1)}_0(\imath s|\mathbf x-\mathbf y|), & \mbox{for $d=2$},\\
\ds\frac{e^{-s|\mathbf x-\mathbf y|}}{4\pi|\mathbf x-\mathbf y|},& \mbox{for $d=3$}.
\end{array}
\right.
\]
(The function $H^{(1)}_0$ is the Hankel function of the first kind and order zero.) We also consider the single and double layer potentials
\begin{eqnarray*}
\mathrm S(s)\lambda := \int_\Gamma \Phi(\punto,\mathbf y;s)\lambda(\mathbf y)\mathrm d\Gamma(\mathbf y) & : & \mathbb R^d\setminus\Gamma \to \mathbb C,\\
\mathrm D(s)\varphi :=\int_\Gamma \partial_{\nu(\mathbf y)}\Phi(\punto,\mathbf y;s)\varphi(\mathbf y)\mathrm d\Gamma(\mathbf y) & : & \mathbb R^d\setminus\Gamma \to \mathbb C
\end{eqnarray*}
and the associated integral operators
\[
\mathrm V(s):=\gamma^+ \mathrm S(s)=\gamma^-\mathrm S(s) \qquad \mathrm K(s):=\smallfrac12 \gamma^+ \mathrm D(s)+\smallfrac12 \gamma^-\mathrm D(s).
\]
Consider then the Laplace transforms of the semidiscrete data $\Phi_h:=\mathcal L\{\varphi_h\}$ and of the semidiscrete solutions $\Lambda_h:=\mathcal L\{ \lambda_h\}$ and $\mathrm U_h:=\mathcal L\{ u_h\}$.
For $z\in \mathbb C_+$ and $\mathrm G \in H^{1/2}(\Gamma)$, we consider the uniquely solvable transmission problem looking for $\mathrm V\in H^1(\mathbb R^d\setminus\Gamma)$ such that
\begin{subequations}\label{eq:6.BB}
\begin{alignat}{4}
z^2 \mathrm V-\Delta \mathrm V = 0 & \qquad & & \mbox{in $\mathbb R^d\setminus\Gamma$}\\
\jump{\gamma \mathrm V} =\mathrm G, & & & \\
\gamma^- \mathrm V\in X_h^\circ, & & & \\
\jump{\partial_\nu \mathrm V}\in X_h. & & & 
\end{alignat}
\end{subequations}

\begin{proposition}\label{prop:6.1} For all $s\in \mathbb C_+$, $\mathrm U_h(s)$ is the unique solution of the transmission problem
\eqref{eq:6.BB} with $z=s$ and $\mathrm G=\Phi_h(s)$. Moreover $\Lambda_h(s)=-\jump{\partial_\nu \mathrm U_h(s)}$.
\end{proposition} 

\begin{proof}
 Note that for all $s\in \mathbb C_+$, $\Lambda_h(s)\in X_h$ and that we have the relationships
\begin{equation}\label{eq:6.1}
\mathrm V(s)\Lambda_h(s)-\smallfrac12 \Phi_h(s) +\mathrm K(s) \Phi_h(s)\in  X_h^\circ, \quad \mathrm U_h(s)=-\mathrm S(s)\Lambda_h(s)-\mathrm D(s)\Phi_h(s).
\end{equation}
This proves the result.
\end{proof}

The CQ discretization affects all four convolutions in \eqref{eq:6.1}. It defines causal functions $u_h^\kappa$ and $\lambda_h^\kappa$ such that $u_h^\kappa(t_n)=u_h^\kappa[n]$ and $\lambda_h^\kappa(t_n)=\lambda_h^\kappa[n]$. Let then
$\Lambda_h^\kappa:=\mathcal L\{ \lambda_h^\kappa\}$ and $\mathrm U_h^\kappa:=\mathcal L\{ u_h^\kappa\}$.

\begin{proposition}\label{prop:6.2} For all $s\in \mathbb C_+$, $\mathrm U_h^\kappa(s)$ is the unique solution of the transmission problem \eqref{eq:6.BB} with 
\[
z=s_\kappa:=\frac{\delta(e^{-s\kappa})}{\kappa}\in \mathbb C_+ \quad \forall s\in \mathbb C_+,
\]
and $\mathrm G=\Phi_h(s)$.
Moreover $\Lambda_h^\kappa(s)=-\jump{\partial_\nu \mathrm U_h^\kappa(s)}$.
\end{proposition}

\begin{proof}
 By construction $\Lambda_h^\kappa(s)\in X_h$ and
\begin{equation}\label{eq:6.3}
\mathrm V(s_\kappa)\Lambda_h^\kappa(s)-\smallfrac12 \Phi_h(s) +\mathrm K(s_\kappa) \Phi_h(s)\in  X_h^\circ, \quad \mathrm U_h^\kappa(s)=-\mathrm S(s_\kappa)\Lambda_h^\kappa(s)-\mathrm D(s_\kappa)\Phi_h(s).
\end{equation}
The result is then straightforward.
\end{proof}

Note that each occurence of $s_\kappa$ in \eqref{eq:6.3} corresponds to the discretization of a convolution process.

We finally consider the errors of time discretization of the semidiscrete-in-space problem
\[
e=e_h^\kappa:=u_h-u_h^\kappa \qquad \mathrm E:=\mathrm U_h-\mathrm U_h^\kappa=\mathcal L\{e\} \qquad \varepsilon=\varepsilon_h^\kappa:=\lambda_h-\lambda_h^\kappa.
\]
On time steps, we will be considering the errors
\[
e_n:=e(t_n)=u_h(t_n)-u_h^\kappa[n], \qquad \varepsilon_n:=\varepsilon(t_n)=\lambda_h(t_n)-\lambda_h^\kappa[n].
\]
Applying Propositions \ref{prop:6.1} and \ref{prop:6.2}, it follows that for all $s\in \mathbb C_+$
\begin{equation}\label{eq:6.5}
\mathrm E(s)\in D_h:=\{ u \in H^1(\mathbb R^d)\cap H^1_\Delta (\mathbb R^d\setminus\Gamma)\,:\, \gamma u\in X_h^\circ, \, \jump{\partial_\nu u}\in X_h\},
\end{equation}
and
\begin{equation}\label{eq:6.6}
s_\kappa^2 \mathrm E(s)-\Delta \mathrm E(s)=(s_\kappa^2-s^2) \mathrm U_h(s)=:\Theta(s).
\end{equation}
In the time domain, the function $\theta=\theta_h^\kappa$, whose Laplace transform is $\Theta$, corresponds to a consistency error of the time discretization --the approximation of the second derivative by the particular  CQ scheme--, applied to the semidiscrete-in-space solution. Before we start the analysis of each of the time discretization methods, let us mention the integration-by-parts formula in $D_h$, which will play an important role in the forthcoming analysis:
\begin{equation}\label{eq:6.7}
(\Delta u,v)_{\mathbb R^d\setminus\Gamma}+(\nabla u,\nabla v)_{\mathbb R^d}=\langle \jump{\partial_\nu u},\gamma v\rangle=0 \qquad \forall u,v\in D_h.
\end{equation}
Bounds with respect to data will be given in terms of the following quantities:
\[
B_k^{1/2}(\varphi_h,t):=\sum_{\ell=3}^k \int_0^t \|\varphi_h^{(\ell)}(\tau)\|_{1/2,\Gamma}\mathrm d\tau \qquad (k\ge 5).
\]
The following product (semi)norm
\[
\triple{(u,v)}:=\Big(\| \nabla u\|_{\mathbb R^d}^2+\| v\|_{\mathbb R^d}^2\Big)^{1/2}
\]
will be used to simplify some formulas.

\subsection{Analysis of the Backward Euler discretization}

\begin{proposition}\label{prop:6.3}
For all $n\ge 1$,
\begin{equation}\label{eq:6.8}
\triple{(e_n,\smallfrac1\kappa (e_n-e_{n-1}))}
\le \kappa \, C t_n(1+t_n) \, B_5^{1/2}(\varphi_h,t_n).
\end{equation}
\end{proposition}

\begin{proof}
Let $f_n:= \frac1\kappa (e_n-e_{n-1})$. Noticing that for the Backward Euler discretization $s_\kappa=\frac1\kappa(1-e^{-s\kappa})$, the error equation \eqref{eq:6.6} can be written as
\begin{equation}\label{eq:6.9}
e_n-e_{n-1}=\kappa f_n, \qquad f_n-f_{n-1}=\kappa \Delta e_n+\kappa \theta_n
\end{equation}
where we can bound the consistency error as
\begin{equation}
\|\theta_n\|_{\mathbb R^d}=\left\|\frac{u_h(t_n)-2u_h(t_{n-1})+u_h(t_{n-2})}{\kappa^2}- \ddot u_h(t_n)\right\|_{\mathbb R^d}
\le \frac{5 \kappa}3 \max_{t_{n-2}\le t\le t_n} \| u_h^{(3)}(t)\|_{\mathbb R^d}.\label{eq:6.10}
\end{equation}
Testing the equations \eqref{eq:6.9} with $-\Delta e_n$ and $f_n$ respectively, adding the results, and applying the integration by parts formula \eqref{eq:6.7} (note that $e_n \in D_h$ for all $n$, since $e$ takes values in this space by \eqref{eq:6.5}), it follows that
\[
(\nabla e_n,\nabla e_n)_{\mathbb R^d}+(f_n,f_n)_{\mathbb R^d}=(\nabla e_n,\nabla e_{n-1})_{\mathbb R^d}+(f_n,f_{n-1})_{\mathbb R^d}+\kappa\, (\theta_n,f_n)_{\mathbb R^d}
\] 	
and therefore
\begin{equation}\label{eq:6.11}
\triple{(e_n,f_n)}\le \triple{(e_{n-1},f_{n-1})}+\kappa\,\|\theta_n\|_{\mathbb R^d} \qquad \forall n.
\end{equation}
Then, by induction
\[
\triple{(e_n,f_n)} \le \kappa \sum_{j=1}^n \| \theta_j\|_{\mathbb R^d}.
\]
Using now \eqref{eq:6.10} and Theorem \ref{th:5.1} (recall that $\varphi \mapsto u_h^\varphi$ is a convolution operator and therefore commutes with differentiation), it follows that
\begin{equation}\label{eq:6.12}
\| \theta_j\|_{\mathbb R^d} \le \frac53\, C_L\,\kappa \, \Big(\max_{t_{j-2}\le t\le t_j}\| \varphi_h^{(3)}(t)\|_{1/2,\Gamma}+ \sqrt{1+C_{t_j}^2} B_2^{1/2}(\varphi_h^{(3)},t_n)\Big),
\end{equation}
where we have also used the fact that the constant $C_t$ of \eqref{eq:4.3} grows with $t$. Adding the bounds \eqref{eq:6.12} for different values of $j$, using the overestimate
\[
\max_{t_{j-2}\le t\le t_j}\| \varphi_h^{(3)}(t)\|_{1/2,\Gamma}\le \max_{0\le t\le t_n}\| \varphi_h^{(3)}(t)\|_{1/2,\Gamma}\le
\int_0^ {t_n}\| \varphi_h^{(4)}(t)\|_{1/2,\Gamma}\mathrm d t,
\]
the result follows.
\end{proof}

\begin{theorem}\label{th:6.4} For all $n\ge 1$
\begin{eqnarray}\label{eq:6.13}
\| e_n\|_{1,\mathbb R^d}&  \le & \kappa \, t_n(1+t_n^2)  B_5^{1/2}(\varphi_h,t_n) \\
\label{eq:6.14}
\| \varepsilon_n\|_{-1/2,\Gamma} &\le & \kappa \,(1+t_n^2) B_6^{1/2}(\varphi_h,t_n). 
\end{eqnarray}
\end{theorem}

\begin{proof} Using the notation of the proof of Proposition \ref{prop:6.3}, and since $e_n=e_{n-1}+\kappa f_n$, we can easily bound (using Proposition \ref{prop:6.3})
\[
\| e_n\|_{\mathbb R^d}\le \kappa \sum_{j=1}^n \| f_j\|_{\mathbb R^d} \le \kappa\,C\,t_n\,(1+t_n)\,B_5^{1/2}(\varphi_h,t_n).
\]
This inequality and Proposition \ref{prop:6.3} prove \eqref{eq:6.13}. 

Note now that $f_n=f(t_n)$, where $f=(e-e(\punto-\kappa))/\kappa$. Therefore, using the second of the equalities \eqref{eq:6.9}, we can bound
\begin{equation}\label{eq:6.15}
\|\Delta e_n\|_{\mathbb R^d\setminus\Gamma}=\left\|\frac1\kappa (f_n-f_{n-1})-\theta_n\right\|_{\mathbb R^d}\le \max_{t_{n-1}\le t\le t_n}\|\dot f(t)\|_{\mathbb R^d}+\|\theta_n\|_{\mathbb R^d}.
\end{equation}
For  $t\in [t_{n-1},t_n]$, we can construct a mesh with time-step $\kappa$ that includes $t$. Then, applying
 Proposition \ref{prop:6.1} with data $\dot\varphi_h$ on this mesh, it follows that
\[
\| \dot f(t)\|_{\mathbb R^d} \le \kappa \,C\,t(1+t)\,B_5^{1/2}(\dot\varphi_h,t) \le \kappa\,C t_n(1+t_n) \,B_5^{1/2}(\dot\varphi_h,t_n).
\]
This inequality, the bound \eqref{eq:6.12} for the consistency error and \eqref{eq:6.15} provide a bound for the Laplacian of the error
\begin{equation}\label{eq:6.16}
\| \Delta e_n\|_{\mathbb R^d\setminus\Gamma} \le \kappa\,C\,(1+t_n^2)\, B_6^{1/2}(\varphi_h,t_n).
\end{equation}
Since by Propositions \ref{prop:6.1} and \ref{prop:6.2} it follows that $\jump{\partial_\nu e_n}=-\varepsilon_n$, the bound \eqref{eq:6.14} is a direct consequence of \eqref{eq:4.7}, Proposition \ref{prop:6.3} and \eqref{eq:6.16}.
\end{proof}

\subsection{Analysis of the BDF2 discretization}

\begin{lemma}\label{lemma:6.5}
Given $\mathrm G \in L^2(\mathbb R^d)$, let $\mathrm V \in D_h$ solve
\[
s_\kappa^2 \mathrm V-\Delta \mathrm V=(s_\kappa^2-s^2)s^{-3}\mathrm G \quad \mbox{in $\mathbb R^d\setminus\Gamma$}
\]
Then
\[
\triple{(\mathrm V, s_\kappa \mathrm V)}
\le C \kappa^2 \frac{|s|}{\min\{1,\mathrm{Re}\,s\}} \|\mathrm G\|_{\mathbb R^d} \quad \forall s\in \mathbb C_+, \quad \forall \kappa, \quad \forall h.
\]
\end{lemma}

\begin{proof}
Consider first a general $z \in \mathbb C_+$, $\mathrm H \in L^2(\mathbb R^d)$ and $\mathrm V \in D_h$ such that $z^2\mathrm V -\mathrm \Delta \mathrm V=\mathrm H$. By \eqref{eq:6.7}, it follows that
\[
z^2 \|\mathrm V\|_{\mathbb R^d}^2+\|\nabla \mathrm V\|_{\mathbb R^d}^2=(\mathrm H,\overline{\mathrm V})_{\mathbb R^d}
\]
and therefore
\begin{equation}\label{eq:6.16}
(\mathrm{Re}\,z)\, \triple{(\mathrm V,z\mathrm V)} \le \|\mathrm H\|_{\mathbb R^d}.
\end{equation}
Taking now $\mathrm H:=(s_\kappa^2-s^2) s^{-3}\mathrm G$, $z=s_\kappa$, and noticing that
\[
|s_\kappa|\le C_1|s|, \qquad |s_\kappa-s|\le C_2 \kappa^2 |s|^3, \qquad \mathrm{Re} s_\kappa \ge C_3 \min\{1,\mathrm{Re}\,s\}, \quad \forall s \in \mathbb C_+, \forall \kappa,
\]
the result follows from \eqref{eq:6.16}.
\end{proof} 

In the next results we will refer to the operator
\[
\partial_\kappa g :=\smallfrac1\kappa (\smallfrac32 g-2 g(\cdot-\kappa)+\smallfrac12 g(\cdot-2\kappa)) \qquad \mathcal L\{\partial_\kappa g\}=s_\kappa \mathrm G(s),
\]
which is the discrete derivative associated to the BDF2 method.

\begin{proposition}\label{prop:6.6}
Let $f:=\partial_\kappa e$ and $f_n:=f(t_n)=\frac1\kappa(\frac32 e_n-2e_{n-1}+\frac12 e_{n-2})$. Then for all $n\ge 1$
\[
\triple{(e_n,f_n)}
 \le \kappa^2\,C\, t_n(1+t_n^2) B_8^{1/2}(\varphi_h,t_n).
\]
\end{proposition}

\begin{proof}
By \eqref{eq:6.6} and Lemma \ref{lemma:6.5} it follows that
\[
\triple{(\mathrm E(s),\mathrm F(s))}
 \le C\kappa^2 \frac{|s|}{\min\{1,\mathrm{Re}\,s\}} \| s^3\mathrm U_h(s)\|_{\mathbb R^d} \qquad s\in \mathbb C_+.
\]
Using then \cite[Theorem 7.1]{DomSaySB}, it follows that
\begin{equation}\label{eq:6.17}
\triple{(e(t),f(t))}
 \le C \kappa^2 t \sum_{\ell=3}^6 \int_0^t\|u_h^{(\ell)}(\tau)\|_{\mathbb R^d}\mathrm d\tau \qquad \forall t.
\end{equation}
Since by Theorem \ref{th:5.1}, we can bound
\begin{equation}\label{eq:6.17b}
\int_0^t \|u_h(\tau)\|_{\mathbb R^d}\mathrm d\tau \le C (1+t^2) B_2^{1/2}(\varphi_h,t),
\end{equation}
the result is a direct consequence of \eqref{eq:6.17}.
\end{proof}

\begin{theorem}\label{th:6.7}
For all $n\ge 1$,
\begin{eqnarray}
\label{eq:6.18}
\| e_n\|_{1,\mathbb R^d} & \le & \kappa^2 \,t_n(1+t_n^3) B_8^{1/2}(\varphi_h,t_n),\\
\label{eq:6.19}
\|\varepsilon_n\|_{-1/2,\Gamma} & \le & \kappa^2\, C \, (1+t_n^3) B_9^{1/2}(\varphi_h,t_n).
\end{eqnarray}
\end{theorem}

\begin{proof}
The proof is very similar to that of Theorem \ref{th:6.4}. Using a simple stability argument for recurrences, we first show that
\[
\| e_n\|_{\mathbb R^d}\le \kappa \sum_{j=1}^n \| f_j\|_{\mathbb R^d}.
\]
This inequality and Proposition \ref{prop:6.6} prove \eqref{eq:6.18}. We next use that $\Delta e_n=(\partial_\kappa f)(t_n)-\theta_n$ (see \eqref{eq:6.6} and the definition of $f$ in Proposition \ref{prop:6.6}) to bound
\begin{eqnarray*}
\|\Delta e_n\|_{\mathbb R^d\setminus\Gamma} & \le & \smallfrac32 \|\smallfrac1{\kappa}(f_n-f_{n-1})\|_{\mathbb R^d}+\smallfrac12 \|\smallfrac1{\kappa}(f_{n-1}-f_{n-2})\|_{\mathbb R^d}+\| (\partial_\kappa^2 u_h-\ddot u_h) (t_n)\|_{\mathbb R^d}\\
& \le & \smallfrac32 \max_{t_{n-2}\le \tau\le t_n} \| \dot f(\tau)\|_{\mathbb R^d} + C \kappa^2 \max_{t_{n-4}\le \tau\le t_n}\| u_h^{(4)}(\tau)\|_{\mathbb R^d}.
\end{eqnarray*}
Using then \eqref{eq:6.17} and \eqref{eq:6.17b} we obtain a bound
\[
\| \Delta e_n\|_{\mathbb R^d\setminus\Gamma} \le \kappa^2 (1+t_n^2) B_9^{1/2}(\varphi_h,t_n)
\]
and \eqref{eq:6.19} follows from this and Proposition \ref{prop:6.6} using that $\varepsilon_n=-\jump{\partial_\nu e_n}$.
\end{proof}

\subsection{Analysis of the trapezoidal rule discretization}\label{sec:6.4}

\begin{proposition}\label{prop:6.8}
For all $n\ge 1$
\[
\triple{(\smallfrac12(e_{n+1}+e_n), \smallfrac1\kappa (e_{n+1}-e_n))}\le \kappa^2\, C \, t_{n+1} B_2^{1/2}(\varphi_h^{(5)}, t_{n+1}).
\]
\end{proposition}

\begin{proof}
In the case of the trapezoidal rule, the error equation \eqref{eq:6.6} can be written as
\[
\smallfrac1{\kappa^2} (1-e^{-s\kappa})^2 \mathrm E(s)=\smallfrac12(1+e^{-s\kappa})^2 \Delta \mathrm E(s)+\smallfrac12{\kappa^2} (1- e^{-s\kappa})^2 \mathrm U_h(s)-\smallfrac12 (1+e^{-s\kappa})^2 s^2\mathrm U_h(s).
\]
In the time domain, this gives
\begin{equation}\label{eq:6.21}
\smallfrac1{\kappa^2} (e_{n+1}-2e_n+e_{n-1})=\smallfrac14\Delta (e_{n+1}+2e_n+e_{n-1}) +\chi_n,
\end{equation}
where
\[
\chi_n:=\smallfrac1{\kappa^2} \big(u_h(t_{n+1})-2u_h(t_n)+u_h(t_{n-1})\big)-\smallfrac12 (\ddot u_h(t_{n+1})+2\ddot u_h(t_n)+\ddot u_h(t_{n-1})).
\]
Using the integration by parts formula \eqref{eq:6.7} we obtain
\begin{eqnarray}\label{eq:6.22}
& &\hspace{-1cm} \Big( \smallfrac1\kappa(e_{n+1}-e_n)-\smallfrac1\kappa(e_n-e_{n-1}), v\Big)_{\mathbb R^d}\\
& & +\smallfrac\kappa2 \Big( \nabla (\smallfrac12(e_{n+1}+e_n)+\nabla (\smallfrac12 (e_n+e_{n-1}),\nabla v\Big)_{\mathbb R^d}=\kappa (\chi_n,v)_{\mathbb R^d} \quad \forall v\in D_h.
\nonumber
\end{eqnarray}
Testing \eqref{eq:6.22} with 
\[
v:=\smallfrac1\kappa(e_{n+1}-e_{n-1})=\smallfrac1\kappa(e_{n+1}-e_n)+\smallfrac1\kappa(e_n-e_{n-1})=\smallfrac2\kappa \Big( \smallfrac12 (e_{n+1}+e_n)-\smallfrac12(e_n+e_{n-1})\Big),
\]
it follows that
\begin{eqnarray}\nonumber
\triple{(\smallfrac12(e_{n+1}+e_n), \smallfrac1\kappa (e_{n+1}-e_n))}^2 &=&
\triple{(\smallfrac12(e_{n}+e_{n-1}), \smallfrac1\kappa (e_{n}-e_{n-1}))}^2\\ 
\nonumber
& & +\kappa\big(\chi_n,\smallfrac1\kappa(e_{n+1}-e_n)+\smallfrac1\kappa(e_n-e_{n-1})\big)_{\mathbb R^d}\\
&=& \kappa \sum_{j=0}^n (\chi_j,\smallfrac1\kappa(e_{j+1}-e_j)+\smallfrac1\kappa(e_j-e_{j-1}))_{\mathbb R^d}.
\label{eq:6.23}
\end{eqnarray}
Consider now the mesh-grid with nodes in the midpoints $t_{j+\frac12}:=(j+\frac12)\kappa$, the piecewise constant function $\widetilde\chi$ such that $\widetilde\chi(t)=\chi_j$ in $(t_{j-\frac12},t_{j+\frac12})$, and the continuous piecewise linear functions $\widetilde f$ and $\widetilde e$ such that
\[
\widetilde f(t_{j+\frac12})=\smallfrac1\kappa(f_{j+1}-f_j),\qquad \widetilde e(t_{j+\frac12})=\smallfrac12(e_{j+1}+e_j).
\]
We can then write \eqref{eq:6.23} as
\[
\triple{ (\widetilde e(t_{n+\frac12}),\widetilde f(t_{n+\frac12}))}^2 = 2\int_0^{t_{n+\frac12}} (\widetilde\chi(\tau),\widetilde f(\tau))_{\mathbb R_d}\mathrm d \tau.
\]
Given $n$, we choose $n^\star\le n$ such that $\triple{ (\widetilde e(t_{n^\star+\frac12}),\widetilde f(t_{n^\star+\frac12}))}$ is maximized. Then
\[
\triple{ (\widetilde e(t_{n^\star+\frac12}),\widetilde f(t_{n^\star+\frac12}))}^2 \le 2 \triple{ (\widetilde e(t_{n^\star+\frac12}),\widetilde f(t_{n^\star+\frac12}))}\int_0^{t_{n^\star+\frac12}} \|\widetilde\chi(\tau)\|_{\mathbb R^d}\mathrm d\tau
\]
and therefore (after comparing with the maximum at $t_{n^\star+\frac12}$ and overestimating the integral in the right-hand side)
\[
\triple{ (\widetilde e(t_{n+\frac12}),\widetilde f(t_{n+\frac12}))} \le 2 \int_0^{t_{n+\frac12}} \|\widetilde\chi(\tau)\|_{\mathbb R^d}\mathrm d\tau\le 2 t_{n+1} \max_{j\le n} \|\chi_j\|_{\mathbb R^d} \qquad \forall n.
\]
Since
\begin{equation}\label{eq:6.24}
\|\chi_n\|_{\mathbb R^d} \le C \kappa^2\max_{t_{n-1}\le \tau\le t_{n+1}}\| u_h^{(4)}(\tau)\|_{\mathbb R^d}, 
\end{equation}
the result follows by \eqref{eq:6.17b}.
\end{proof}

\begin{theorem}\label{th:6.9}
For all $n\ge 1$,
\begin{eqnarray}\label{eq:6.25}
\| e_n\|_{\mathbb R^d} &\le & \kappa^2 C t_n^2 B_7^{1/2}(\varphi_h,t_n),\\
\|\smallfrac14(\varepsilon_{n+1}+2\varepsilon_n+\varepsilon_{n-1})\|_{-1/2,\Gamma} & \le & \kappa^2 C (1+t_{n+1}^2) B_9^{1/2}(\varphi_h,t_{n+1}).
\end{eqnarray}
\end{theorem}

\begin{proof}
Since
\[
e_n=\kappa \sum_{j=0}^{n-1} \frac1\kappa(e_{j+1}-e_j),
\]
the first bound follows from Proposition \ref{prop:6.8}. By \eqref{eq:6.21} we can bound
\begin{equation}\label{eq:6.27}
\|\smallfrac14\Delta (e_{n+1}+2e_n+e_{n-1})\|_{\mathbb R^d\setminus\Gamma} \le C  \max_{t_{n-1}\le \tau\le t_{n+1}} \|\ddot e(\tau)\|_{\mathbb R^d} +\|\chi_n\|_{\mathbb R^d}
\end{equation}
The first term in the right-hand-side of \eqref{eq:6.27} can be bounded using \eqref{eq:6.25} applied to $\ddot\varphi_h$ using a time-grid with time-step $\kappa$ that includes the point where the maximum is attained. The second term of \eqref{eq:6.27} is bounded using \eqref{eq:6.24}. The result follows from the fact that $\varepsilon_n=-\jump{\partial_\nu e_n}$.
\end{proof}

\section{Final comments}

In this paper we have given a full analysis of the discretization with Galerkin is space and three particular instances of Convolution Quadrature in time of a direct formulation for the exterior Dirichlet problem for the wave equation. The full error estimates are the result of Theorem \ref{th:3.1} for the semidiscretization in space process and Theorems \ref{th:6.4} (backward Euler), \ref{th:6.7} (BDF2), and \ref{th:6.9} (trapezoidal rule) for time discretization.

An indirect formulation, i.e., 
\[
\mathcal V * \xi = \varphi \qquad u=\mathcal S * \xi
\]
follows from very similar arguments. The Galerkin projection is the same and therefore, the analysis of the semidiscrete in psace system is a particular case of the results in this paper. The Galerkin solver (see Section \ref{sec:5}) is however slightly different. While its analysis is not needed for the semidiscretization in space, it is needed for the time discretization. This analysis is likely to be extremely similar to the one given here. In terms of its Laplace transform, the semidiscrete problem is
\begin{alignat*}{4}
s^2 \mathrm U_h(s)-\Delta \mathrm U_h(s) = 0 & \qquad & & \mbox{in $\mathbb R^d\setminus\Gamma$}\\
\jump{\gamma \mathrm U_h(s)} =0, & & & \\
\gamma \mathrm U_h(s)-\Phi(s) \in X_h^\circ, & & & \\
\jump{\partial_\nu \mathrm U_h(s)}\in X_h. & & & 
\end{alignat*}
with $\Xi_h(s)=\jump{\partial_\nu \mathrm U_h(s)}$. This is a very similar problem (same kind of transmission conditions) to problem \eqref{eq:6.BB} (see also Proposition \ref{prop:6.1}). In particular, the error equations to compare semidiscrete and fully discrete solutions \eqref{eq:6.5}-\eqref{eq:6.6} are the same as in the case of the direct formulation and all the arguments of Section \ref{sec:6} hold, contingent to having proved the estimates of Theorem \ref{th:5.1} adapted to the new kind of Galerkin solver. 

All the arguments that have been used in this paper can be easily extended to the case of the single layer potential for the elastic wave equation in any dimension.

Much of the analysis of Sections \ref{sec:4}-\ref{sec:6} can be done using estimates in the Laplace domain. That gives a more streamlined way of proving estimates, although they come with either worse constants (for growth in time) or with higher continuity requirements: see \cite[Section 7]{DomSaySB} for a comparison of Laplace domain and time domain techniques applied to estimating layer potentials and integral operators. The analysis of semidiscretization in space using the Laplace domain can be adapted from the techniques developed in \cite{LalSay09}. Analysis of convolution quadrature can then be carried out using the very general results of Lubich \cite{Lub94} applied to the semidiscrete operators. It has to be noted, though, that the analysis in \cite{Lub94} does not cover the case of the trapezoidal rule (the reference \cite{Ban10} circunvents this difficulty nevertheless), while the relatively traditional time-domain analysis of Section \ref{sec:6.4} --based on understading the semidiscrete equations as a transmission problem and, in particular, on the integration by parts formula \eqref{eq:6.7}--, is applicable. Similarly, the use of multistage convolution quadrature  \cite{BanLub11}, \cite{BanLubMel11} and variable-step convolution quadrature  \cite{LopSauSB} can be applied using estimates in the Laplace domain and it remains to be seen whether a time-domain analysis is practicable and produces different or improved results.

\bibliographystyle{abbrv}
\bibliography{RefsTDIE}

\begin{thebibliography}{10}

\bibitem{BamHaD86a}
A.~Bamberger and T.~H. Duong.
\newblock Formulation variationnelle espace-temps pour le calcul par potentiel
  retard\'e de la diffraction d'une onde acoustique. {I}.
\newblock {\em Math. Methods Appl. Sci.}, 8(3):405--435, 1986.

\bibitem{BamHaD86b}
A.~Bamberger and T.~H. Duong.
\newblock Formulation variationnelle pour le calcul de la diffraction d'une
  onde acoustique par une surface rigide.
\newblock {\em Math. Methods Appl. Sci.}, 8(4):598--608, 1986.

\bibitem{Ban10}
L.~Banjai.
\newblock Multistep and multistage convolution quadrature for the wave
  equation: algorithms and experiments.
\newblock {\em SIAM J. Sci. Comput.}, 32(5):2964--2994, 2010.

\bibitem{BanLub11}
L.~Banjai and C.~Lubich.
\newblock An error analysis of {R}unge-{K}utta convolution quadrature.
\newblock {\em BIT}, 51(3):83--496, 2011.

\bibitem{BanLubMel11}
L.~Banjai, C.~Lubich, and J.~M. Melenk.
\newblock Runge-{K}utta convolution quadrature for operators arising in wave
  propagation.
\newblock {\em Numer. Math.}, 119(1):1--20, 2011.

\bibitem{BanSauSB}
L.~Banjai and S.~Sauter.
\newblock Frequency explicit regularity estimates for the electric field
  integral operator.
\newblock Preprint I-Math, Universit\"at Z\"urich, 05-2012.

\bibitem{BanSchTA}
L.~Banjai and M.~Schanz.
\newblock Wave propagation problems treated with convolution quadrature and
  {BEM}.
\newblock In Langer, M.~Schanz, O.~Steinbach, and W.~Wendland, editors, {\em
  Fast Boundary Element Methods in Engineering and Industrial Applications},
  page 145–187.

\bibitem{DomSaySB}
V.~Dom{\'\i}nguez and F.-J. Sayas.
\newblock Some properties of layer potentials and boundary integral operators
  for the wave equation.
\newblock To appear in J. Integral Equation Appl.

\bibitem{LalSay09b}
A.~R. Laliena and F.-J. Sayas.
\newblock A distributional version of {K}irchhoff's formula.
\newblock {\em J. Math. Anal. Appl.}, 359(1):197--208, 2009.

\bibitem{LalSay09}
A.~R. Laliena and F.-J. Sayas.
\newblock Theoretical aspects of the application of convolution quadrature to
  scattering of acoustic waves.
\newblock {\em Numer. Math.}, 112(4):637--678, 2009.

\bibitem{LopSauSB}
M.~L\'opez~Fern\'andez and S.~Sauter.
\newblock A generalized convolution quadrature with variable time stepping.
\newblock Preprint I-Math, Universit\"at Z\"urich, 17-2011.

\bibitem{Lub94}
C.~Lubich.
\newblock On the multistep time discretization of linear initial-boundary value
  problems and their boundary integral equations.
\newblock {\em Numer. Math.}, 67(3):365--389, 1994.

\bibitem{McL00}
W.~McLean.
\newblock {\em Strongly elliptic systems and boundary integral equations}.
\newblock Cambridge University Press, Cambridge, 2000.

\bibitem{Paz83}
A.~Pazy.
\newblock {\em Semigroups of linear operators and applications to partial
  differential equations}, volume~44 of {\em Applied Mathematical Sciences}.
\newblock Springer-Verlag, New York, 1983.

\bibitem{Ryn99}
B.~P. Rynne.
\newblock The well-posedness of the electric field integral equation for
  transient scattering from a perfectly conducting body.
\newblock {\em Math. Methods Appl. Sci.}, 22(7):619--631, 1999.

\bibitem{SaySB}
F.-J. Sayas.
\newblock Energy estimates for {G}alerkin semidiscretizations of time domain
  boundary integral equations.
\newblock To appear in Numer. Math.

\bibitem{Sch01}
M.~Schanz.
\newblock {\em Wave Propagation in Viscoelastic and Poroelastic Continua: A
  Boundary Element Approach (Lecture Notes in Applied and Computational
  Mechanics)}.
\newblock Springer, 2001.

\end{thebibliography}

\end{document}